\documentclass{mcom-l}
\usepackage[ruled,vlined]{algorithm2e}
\usepackage{algpseudocode}
\algrenewcommand\algorithmicrequire{\textbf{Input:}}
\algrenewcommand\algorithmicensure{\textbf{Output:}}
\usepackage[normalem]{ulem}
\usepackage{hyperref}[showkeys]
\usepackage{xcolor}
\usepackage{graphicx}
\usepackage{enumerate}
\usepackage{amsrefs}
\usepackage{cleveref}

\newtheorem{theorem}{Theorem}[section]
\newtheorem{lemma}[theorem]{Lemma}

\theoremstyle{definition}

\newtheorem{asm}{Assumption}
\newtheorem{col}[theorem]{Corollary}

\theoremstyle{remark}
\newtheorem{remark}[theorem]{Remark}

\numberwithin{equation}{section}

\newcommand{\comment}[1]{{} }
\begin{document}

\title[Stochastic Self-Consistent Calculations]{Stochastic Algorithms for Self-consistent Calculations of Electronic Structures}


\author{Taehee Ko}
\address{Department of Mathematics \\
  The Pennsylvania State University, University Park, PA 16802, USA\\}
\email{tuk351@psu.edu}

\author{Xiantao Li}
\address{Department of Mathematics \\
  The Pennsylvania State University, University Park, PA 16802, USA\\}
\email{Xiantao.Li@psu.edu}
\thanks{This work is supported by NSF Grants DMS-1819011 and 1953120.}


\date{}

\dedicatory{}

\begin{abstract}
The convergence property of a stochastic algorithm for the self-consistent field (SCF) calculations of electron structures is studied. The algorithm is formulated by rewriting the electron charges as a trace/diagonal of a matrix function, which is subsequently expressed as a statistical average. The function is further approximated by using a Krylov subspace approximation. As a result, each SCF iteration only samples one random vector without having to compute all the orbitals. We consider the common practice of SCF iterations with damping and mixing. We prove that the iterates from a general linear mixing scheme converge in a probabilistic sense when the stochastic error has a second finite moment.
\end{abstract}

\maketitle

\section{Introduction}
The computation of electron structures has recently become routine calculations in material science and chemistry \cite{martin2004@book}. Many software packages  have been developed to facilitate  these efforts \cites{elstner1998self,kronik2006parsec,marques2003,soler2002siesta}. 
A central component in modern electronic-structure calculations is the self-consistent  field (SCF) calculations \cites{parr1995density,martin2004@book}. The standard procedure is to start with a guessed density, and then determine the Hamiltonian, followed by the computation of the eigenvalues and eigenvectors  which lead to a new density; The procedure continues until the input and output densities are close.  Many numerical methods have been proposed to speed up the SCF procedure, see  \cites{banerjee2016periodic,bowler2000efficient,cances2001self,cances2000can,cances2000convergence,hamilton1986direct,kresse1996efficient,johnson1988modified,lin2013elliptic,fang2009two,payne1992iterative,zhou2006self,yang2006constrained,zhang2014gradient}. Overall, the SCF still dominates the computation, mainly because of the unfavorable cubic scaling in the computation of the eigenvalue problem.   SCF is also a crucial part of ab initio calculations, especially in the Born-Opennheimer molecular dynamics \cites{marx2009ab,tuckerman2002ab}: The motion of the nuclei causes the external potential to change continuously, and the SCF calculations have to be performed at each time step. 

The SCF problem can be formulated as a fixed-point iteration (FPI). One remarkable, but much less explored approach for FPIs, is the random methods \cite{alber2012stochastic}, which are intimately connected to the stochastic algorithms of Robbins and Monro \cites{robbins1951stochastic,chung1954stochastic,wolfowitz1952stochastic}, which in the context of machine learning, has led to the stochastic gradient descent (SGD) methods \cite{bottou2018optimization}. The advantage of these stochastic methods is for optimzation problems with very large data set,  one only calculates a small subset of samples  rather than the entire set.

The main purpose of this paper is to formulate such a stochastic algorithm in the context of SCF, and analyze its convergence property.
We first propose to use the diagonal estimator  \cite{bekas2007estimator} to approximate the matrix function involved in the SCF. The key observation is that with such a diagonal estimator, the approximate fixed-point function can be expressed as a conditional expectation. Consequently, we construct a random algorithm, where we choose a random vector to sample the conditional average. What bridges these two components together is the Krylov subspace method \cite{saad1992analysis}  that incorporates the random vector as the starting vector and approximates the matrix function using the Lanczos algorithm.
In light of the importance of mixing methods in SCF
\cites{banerjee2016periodic,bowler2000efficient,hamilton1986direct,kresse1996efficient,johnson1988modified}, which often enable and speed up the convergence of the fixed-point iterations, we consider iterative methods with damping 
and mixing, together with the stochastic algorithm. We also  present preliminary numerical results based on a generalized linear mixing scheme.

 
The convergence of SCF is certainly an outstanding challenge in scientific computing. But our analysis is applicable to general stochastic fixed-point problems. In particular, our convergence analysis 
treats general stochastic fixed-point problems when the sampling error only has a second moment bound. In this setting, one-step iteration algorithms can be viewed as a Markov chain.  Kushner and Yin \cites{kushner1965stability,kushner1967stochastic,kushner2003stochastic} introduced the notion of stochastic stability of  discrete-time Markov chains and proved their convergence with probability one when such stability holds. The underlying idea is similar to the Lyapunov function theory for ODEs.  The Markov chains considered in  \cites{kushner2003stochastic} has a similar form as the simple mixing scheme. Motivated by their analysis, we shall prove stochastic stability and convergence of the simple mixing scheme.  However, such an approach can not be directly extended to general mixing schemes, which correspond to high order Markov chains. In order to overcome this difficulty, we generalize the Lyapunov functions for extended Markov chains. Remarkably, from this, one can interpret the general mixing scheme as a first-order Markov chain. Further, we establish tools that link the convergence of the extended Markov chains to the properties of the iterates which are of our interest. 
 With the tools and the generalized Lypaunov functions, we will prove that general mixing schemes are also stable and converge to the fixed point with probability one when stability holds, still under the milder condition that the second moment of the stochastic error is finite.

In practice, the models for electronic structure calculations, e.g., the density-functional theory (DFT) \cites{Kohn1965,hohenberg1964inhomogeneous}, has to be discretized in space. One straightforward implementation is the real space discretization using finite difference \cite{beck_real-space_2000}. To illustrate how to formulate a stochastic fixed-point problem from a more sophisticated discretization, we consider the framework of the self-consistent charge density functional tight-binding method (SCC-DFTB) \cite{elstner1998self}, which has been an important semi-empirical methodology in modern  electronic structure simulations. More specifically, Elstner and coworkers  devised the method as an improvement of the non-SCC approach. As an application, we will present our stochastic algorithms based on this tight-binding framework, although the application to real-space methods , e.g., \cites{beck_real-space_2000,kronik2006parsec,suryanarayana2010non}, is straightforward. In the computational chemistry literature, the stochastic DFT  \cite{cytter2018stochastic} shows resemblance to the present work, especially with the use of randomized algorithms for estimating traces and approximation methods for matrix-vector products. The trace estimator in their work is essentially equivalent to the diagonal estimator \cites{bekas2007estimator}, which is also used in this work. In \cites{bekas2007estimator} the density-matrix is approximated by  Chebyshev polynomials, while in our approach,   we use the Krylov subspace approximation. 
More importantly, the framework proposed in \cites{bekas2007estimator} is mainly computational. In contrast, this paper presents several theoretical results that are crucial to understanding the performance of such algorithms. Another class of methods that also work with the density matrix is linear-scaling algorithms for DFT \cite{bowler2002recent}, which is achieved by exploiting the sparsity of the density matrix. 

The rest of the paper is organized as follows.  We first present  general  stochastic fixed-point problems and mixing schemes. Section \ref{section3} presents convergence analysis, with emphasis on convergence in probability and the implication to computational complexity.  
In Section \ref{section4}, we focus specifically on electron structure calculations. We review the SCF procedure in a tight-binding model \cite{elstner1998self}. We show that the electron charges can be expressed as a trace formula. Based on such expressions, we construct a stochastic algorithm, using the methods in Section \ref{section2}. In Section \ref{section5}, we present some numerical results.

\section{Stochastic Fixed-Point Iterations and Mixing schemes}\label{section2}
 We consider numerical methods for the fixed-point problem,
\begin{equation}\label{eq: problem}
    q=K(q)=\mathbb{E}[k(q,v)].
\end{equation}
Here $q\in\mathbb{R}^N$, and the mapping $K:\mathbb{R}^N\mapsto\mathbb{R}^N$,  is represented as the expectation of some random mapping $k(q,v)$ with respect to a random vector $v$ whose distribution is known in advance. This problem will be referred to as a \emph{stochastic fixed-point problem}.

A direct approach for the problem \eqref{eq: problem} is the fixed-point iterations,
\begin{equation}
    q_{n+1} = K(q_n).
\end{equation}
More generally, one can introduce damping and mixing to improve the convergence, as follows,
\begin{equation}\label{eq: lmixing}
     q_{n+1}=(1-a_n)\sum_{i=1}^mb_iq_{n-m+i}+a_n\sum_{i=1}^mb_iK(q_{n-m+i}).
\end{equation}
Here $a_n \in (0,1)$ can be regarded as a damping parameter;  $m\in \mathbb{N}$, and the $m$ constants $b_1,b_2,..,b_m \in \mathbb{R}$ satisfy the conditions that $\sum_{i=1}^mb_i=1$ and $b_i\ge 0$.  

Meanwhile, directly implementing the scheme \eqref{eq: lmixing} for the stochastic fixed-point problem \eqref{eq: problem} requires repeated sampling of $k(q,v)$ \cite{toth_local_2017}, which can be computationally demanding. The stochastic algorithm addresses this problem by
\begin{equation}\label{eq: linearm}
        q_{n+1}=
        (1-a_n) \sum_{i=1}^mb_iq_{n-m+i}+a_n\sum_{i=1}^mb_ik(q_{n-m+i},v_{n-m+i}).
\end{equation}
Namely, we only sample the random fixed-point function once (or a small number of times) in each iteration as shown {\bf Algorithm} \ref{alg: linear}. 

To better describe the linear mixing method \eqref{eq: linearm},
we denote by
\begin{equation}\label{eq: Bm}
 B_m(y_n):=\sum_{i=1}^mb_iy_{n-m+i},   
\end{equation}
the linear combination of a number of previous iterations.  The terminology of {\it linear} mixing simply means that the right-hand side forms a linear combination of the previous iterates.  As outlined in Algorithm \ref{alg: linear}, the implementation is quite straightforward. 

\begin{algorithm}
\SetAlgoLined
	\KwData{$\{a_n\}$, $\{b_i\}_{i=1}^m$,$\{q_i\}_{i=1}^m$}
	\KwResult{Approximate fixed-point}
	\label{alg: linear}
	\For{$n=m,m+1,...,\;$until convergence} 
    {Compute $k(q_{n+1},v_{n_+1})$ \;
    $q_{n+1}=(1-a_n)B_m(q_{n})+a_n B_m(k(q_n,v_n))$\;}
    	\caption{Linear mixing method}
\end{algorithm}

When $m=1$, the linear mixing scheme \eqref{eq: simplem} is reduced to
 \begin{equation}\label{eq: simplem}
        q_{n+1}=
        (1-a_n)q_{n}+a_n k(q_{n},v_{n}),
\end{equation} 
which is known as simple mixing in electronic structure calculations.  Under appropriate assumptions, the iterations from the simple mixing \eqref{eq: simplem}  have been shown to converge almost surely \cites{alber2012stochastic}. Similar results are established for the stochastic gradient descent (SGD), in the machine learning literature \cite{bottou2018optimization}, which was originated from the work \cites{robbins1951stochastic}. However, an important question is whether the linear mixing method \eqref{alg: linear} converges with a general depth $m\geq 1$ under mild assumptions. In the following section, we show that despite the random noise within the sample $k(q,v)$, the mixing scheme \ref{alg: linear} for any $m\geq 1$ converges in the probabilistic sense.

\begin{remark}
  The mixing methods require multiple initial guesses. They can be computed from the simple mixing method \eqref{eq: simplem}. Alternatively,  this can be done by setting $m=1$ to generate the second iteration $q_2,$ and then $m=2$ to find $q_3$, until all the initial vectors are computed, see \cite{toth_convergence_2015}. For simplicity, we assume that all the $m$ initial vectors have been computed, and our analysis focuses on the subsequent iterations. 
\end{remark}

\begin{remark} In practice, there are  situations where  $K(q)$ is approximated  by  $K_{\ell}(q)$, which is easier to compute. The parameter $\ell$ indicates the order of such  anapproximation. Compared to the original problem \eqref{eq: problem}, this approximation leads to a perturbed fixed-point problem,
 \begin{equation*}
     q=K_{\ell}(q).
 \end{equation*}In Section \ref{section4} we will quantify the effect of such perturbation in the context of electron structure calculations with Theorem \ref{krylovapprox} and the numerical results in Figure \ref{fig:errvsd}. 
 
\end{remark}

\section{Convergence Theorems and Complexity Estimates}\label{section3}
 In this section, we present convergence analysis for the linear mixing scheme \eqref{alg: linear}, which are applicable to a large class of stochastic fixed-point problems. Under standard assumptions, these theorems highlight the stochastic stability properties and probabilistic convergence of the mixing algorithm. 

Let us first outline the main ingredients in the proofs at the high level. 
To characterize stochastic stability, we define a  family of Lyapunov functions whose input contain the $m$ iterates and the intermediate stochastic errors. With the results in Appendix \ref{sec: A5}, we will show that the extended Markov chains will produce non-negative supermartingales with the family of Lyapunov functions with some perturbations. In the end, we employ the optional stopping theorem on non-negative supermartingales to complete the proof  \cites{williams1991probability,kushner2003stochastic,resnick2019probability}. 


\subsection{Assumptions}\label{sectionsto}
 The first assumption is that the mapping $K$ is \emph{locally} contractive, since many problems in practice, including electron structure calculations, are not expected to be globally contractive.  Remarkably, as we show after Theorem \ref{casem}, this condition can be further relaxed to a stability condition associated with the Jacobian of the mapping $K$ at its fixed-point $q^*$, under which the following convergence theorems still hold.

\begin{asm}\label{as:04}
 For some $\rho>0$ and some $c\in(0,1)$,
\begin{equation}\label{eq: contractive}
\begin{aligned}
    &\|K(x)-K(y)\|_2\leq c\|x-y\|_2
\end{aligned}
\end{equation}for all $x,y\in$ $\mathcal{B}(q^*,\rho)$ where $\mathcal{B}(q^*,\rho)$ is the ball centered at $q^*$ of radius $\rho$ with respect to the $2$-norm.
\end{asm}

\medskip 

The second assumption is on the random error arisen in the evaluation $K(q)$.
\begin{asm}\label{as:05}
For every $q\in \mathcal{B}(q^*,\rho)$, the random error $\xi$
\begin{equation}\label{eq: xi}
    \xi(q,v):=k(q,v)-K(q), 
\end{equation}
 satisfies,
\begin{equation}
        \mathbb{E}\big[\xi(q,v)|q\big]=0\quad\text{and}\quad    \sup_{q\in B(q^*,\rho)}\mathbb{E}\big[\|\xi(q,v)\|_2^2|q\big]\leq\Xi.
\end{equation}
\end{asm}
 The first condition means that $\xi(q,v)$ has zero mean, or the approximation by $k(q,v)$ is unbiased.  The second condition states that the variance of the random error is uniformly bounded in the domain of $K$ defined in Assumption \ref{as:04}. These conditions are standard in the machine learning theory \cite{bottou2018optimization}. In terms of application, Assumption \ref{as:05} is fulfilled by the construction of a stochastic algorithm for the DFTB+ as we will explain in Corollary \ref{noise}.

\subsection{Stochastic Stability and Probabilistic Convergence}

We begin by introducing notations for the following theorems. $\mathbb{I}_A$ denotes the indicator function which values one in the event $A$ and zero otherwise. Let $R_n$ be the residual equal to $K(q_n)-q_n$ and $\xi_n:=\xi(q_n,v_n)$ be the stochastic error at step $n$. We define
\begin{equation}\label{eq: residual}
G_n:=R_n+\xi_n    
\end{equation}as the sum of the residual and the stochastic error at step $n$. 

Let us first establish the convergence of the simple mixing scheme. Motivated by the analysis in \cite{kushner2003stochastic}, we start by defining a perturbed Lyapunov functional,  
\begin{equation}\label{simLya}
V_n(q_n)=V(q_n)+\delta V_n,    
\end{equation}
where
\begin{equation}\label{eq: V}
    V(q)=\|q-q^*\|_2^2, \quad \delta V_n=\Xi\sum_{i=n}^{\infty}a_i^2.
\end{equation}

\begin{theorem}\label{stabsimple}
Assume that the damping parameters $\{a_n\}$ satisfies \begin{equation}\label{eq: conda}
   \sum_n a_n = \infty, \quad   \sum_n a_n^2<\infty,\quad a_n\leq \frac{1-c}{(1+c)^2}\;\; \forall n\in\mathbb{N}. 
\end{equation}Under Assumptions \ref{as:04} and \ref{as:05}, the simple mixing scheme \eqref{eq: simplem} $(m=1)$ has the following properties:
 \begin{enumerate}
     \item [{\bf (i)}] The iterations $\{q_n\}_{n=1}^{\infty}$ leave the ball $B(q^*,\rho)$ with probability,
     \begin{equation*}
         \mathbb{P}\left\{\sup_{n}\|e_n\|_2 > \rho|q_1\right\}\mathbb{I}_{\{q_1\in B(q^*,\rho)\}}\leq \frac{V_1(q_1)}{\rho^2},
     \end{equation*}
where the function $V_1$ is given as \eqref{simLya}. 

Each path $\{q_n\}_{n=1}^{\infty}$ that stays in the ball will be called a stable path.      
     \item [{\bf (ii)}] Each stable path converges to $q^*$, i.e.,
\begin{equation*}
    \mathbb{P}\left\{\lim_{n\to\infty}q_n=q^*|\{q_n\}_{n=1}^{\infty}\subset B(q^*,\rho)\right\}=1.
\end{equation*}
 \end{enumerate}
 Consequently, the iterations $\{q_n\}_{n=1}^{\infty}$ converge to $q^*$ with probability at least $1-\frac{V_1(q_1)}{\rho^2}$.
\end{theorem}
To highlight the main theoretical results, the proof is included in Appendix \ref{sec: A6}.

\medskip

To handle the linear mixing scheme \eqref{alg: linear} with $m\geq 2$, we work with the $m$ vectors at each step. Motivated with the framework and notations in \cite{kushner2003stochastic}, we define an extended state variable by lumping every $m$ iterations  coupled with the stochastic noises
\begin{equation}
    X_n:=(q_{n+m-1},\xi_{n+m-2},q_{n+m-2},..,\xi_n,q_{n})\in\mathbb{R}^{(2m-1)N},
\end{equation} which forms a first-order Markov chain. In accordance with this, we consider a filtration $\{\mathcal{F}_n\}$ which measures at least $\{X_i,i\leq n\}$. Let us denote by $\mathbb{E}_n$ the expectation conditioned on $\mathcal{F}_n$. 

We assume that the mixing coefficient satisfies that
\begin{equation}\label{mixingcoeff}
    \sum_{i=1}^mb_i=1,\quad b_i\geq 0,\quad b_m>0.
\end{equation}

The general case $m\geq 2$ requires a more sophisticated construction of the Lyapunov function. We define a more general Lypaunov function as 
\begin{equation}
    \|X_n\|_n=\|e_{n+m-1}\|_2^2+\sum_{j=2}^m\sum_{i=j}^mb_{m-i+1}\|e_{n+j-2+m-i}+a_{n+m-3+j}G_{n+j-2+m-i}\|_2^2.
\end{equation}The subscript $n$ in $\|\cdot\|_n$ is meaningful, since this Lyapunov function depends on $n$ as the coefficient $a_{n+m-3+j}$ in each summand does so.

Together with this function, we define the perturbed Lyapunov function
\begin{equation}\label{mLya}
    V_n(X_n)=\|X_n\|_n+\Xi \sum_{i=n}^{\infty}\chi_i,
\end{equation}where 
\begin{equation}\label{chi-n}
    \chi_n:=\sum_{j=1}^mb_{m-j+1}a_{n+m-2+j}^2.
\end{equation}Note that $\chi_n$ is recognized as a weighted sum of the $m$ squares of the damping parameters. If $m=1$, it is exactly the case of simple mixing in \eqref{eq: V}.

\begin{theorem}\label{casem}
Assume that the damping parameters $\{a_n\}$ satisfy \eqref{eq: conda} and the mixing coefficients satisfy \eqref{mixingcoeff}. Under Assumptions \ref{as:04} and \ref{as:05}, the general mixing scheme \eqref{eq: linearm} ($m\geq 2$) satisfies
\begin{enumerate}
    \item [\bf (i)]The iterations $\{q_n\}_{n=1}^{\infty}$ leaves the ball $B(q^*,\rho)$ with  probability bounded by, 
    \begin{equation*}
    \mathbb{P}\left\{\sup_{n\geq m+1}\|e_n\|_2>\rho|\{q_i\}_{i=1}^m\right\}\mathbb{I}_{\{\{q_i\}_{i=1}^m\subset B(q^*,\rho)\}}\leq\frac{V_1(X_1)}{\rho^2},
\end{equation*}where $V_1$ is given as \eqref{mLya}.
    \item [\bf (ii)]Each stable path converges to $q^*$,
    \begin{equation*}
        \mathbb{P}\left\{\lim_{n\to\infty}q_n=q^*|\{q_n\}_{n=1}^{\infty}\subset B_{\infty}(q^*,\rho)\right\}=1.
    \end{equation*}
\end{enumerate}
 Consequently, the linear mixing scheme $(m\geq 2)$ converges to $q^*$ with probability at least $1-\frac{V_1(X_1)}{\rho^2}$.
\end{theorem}
The proof for general $m$ is similar to that of Theorem \ref{stabsimple}. We refer readers to the proof in Appendix \ref{sec: A6}.

\bigskip

The contractive property \eqref{eq: contractive} clearly plays an important role in the analysis of fixed-point iterations. But this starting point is not specific to a stochastic algorithm. Rather, it is also essential in deterministic settings \cites{toth_convergence_2015,toth_local_2017}. In the context of SCF iterations, \cites{lin2013elliptic,cances2021convergence} showed that with a damping term, the contractive property \eqref{eq: contractive} of the modified map, 
\begin{equation}\label{Ktheta}
    K_\theta:= (1-\theta)q + \theta K(q),
\end{equation}
can be guaranteed when the eigenvalues of the Jacobian of the fixed-point map at $q^*$, here denoted by $K'(q^*)$, are real and less than $1$. This implies that there exists a $\theta_{\max}>0$, such that whenever $0<\theta<\theta_{\max}$, the spectral radius of $  K'_\theta(q^*)$ is less than 1.  The argument in \cite{lin2013elliptic} used the connection to the dieletric operator that can be symmetrized, and its physical interpretation of material stability. Similar stability conditions have been used in \cite{cances2021convergence} to prove convergence of iterative methods for optimization problems. 

 We found that 
it is enough to assume that $K'(q^*)$ has eigenvalues in $\mathbb{C}$ with real part less than 1 to ensure  the local contraction  \eqref{eq: contractive} under some vector norm. This is more general than the conditions in  \cites{lin2013elliptic,cances2021convergence},  as indicated in the following theorem,
\begin{theorem}\label{thmcontract'}
Suppose that $K$ is continuously differentiable and $\mathrm{Re}(\lambda)<1$ for each eigenvalue $\lambda$ of $K'(q^*)$. Then, there exists an unitary transformation $U$ and a vector norm, $\| x \|_D= (x,Dx)^{1/2}$ for some $D \succ 0,$ and  under this norm  the mapping 
$K_{U,\theta}(p)= (1-\theta) p + \theta U^{-1}K(Up)  $
is locally contractive 
for any $0< \theta < \theta_{\max}$.
\end{theorem}
The proof can be found in Appendix \ref{sec: A2}.



Now we demonstrate how the results in Theorems \ref{stabsimple} and \ref{casem} can still be retained under such a relaxed condition. Without loss of generality, we consider the simple mixing scheme. Choose $\theta>0$ such that the mapping $K_{U,\theta}$ is contractive according to Theorem \ref{thmcontract'}. Let $a_n$ be any damping coefficient such that $\theta>a_n>0$. Using the unitary transformation $U$, one can obtain equivalent iterations using the  follows steps,
\begin{subequations}\label{modification}
\begin{eqnarray}
        &q_{n+1}=&(1-a_n)q_n+a_n\big(K(q_n)+\xi_n\big),\\
        &p_{n+1}=&(1-a_n)p_n+a_n\big(K_U(p_n)+U^{-1}\xi_n\big),\\
        &p_{n+1}=&(1-\frac{a_n}{\theta})p_n+\frac{a_n}{\theta}\big(K_{U,\theta}(p_n)+\theta U^{-1}\xi_n\big). \label{eq: pn}
\end{eqnarray}
\end{subequations}The first line is simple mixing with a mapping $K$. By multiplying $U^{-1}$ from the left, the second expression follows by defining $K_U(p):=U^{-1}K(Up)$. More importantly, the third expression can be viewed as the simple mixing scheme with the contraction $K_{U,\theta}$ by Theorem \ref{thmcontract'}. 
Moreover, the constant factor $\theta U^{-1}$ in front of the error $\xi_n$ does not affect the assumption \ref{as:05}. Therefore, the analysis in Theorems \ref{stabsimple} and \ref{casem} can be applied to \eqref{eq: pn}  to obtain the same result in terms of $\{p_n\}$, which can be extended to the iterations $\{q_n\}$ using the equivalence of norms.   The same idea can be applied to the general linear mixing scheme \eqref{eq: linearm} due to the linear combination of the fixed-point functions.


\subsection{The iteration complexity of the stochastic linear mixing method}\label{cmp}

In this section, we study the iteration complexity, that is, the number of required iterations to reach an accuracy threshold $\epsilon$. In terms of stochastic approximation methods, the iteration complexity has been an important topic in optimization problems \cite{bottou2018optimization}.

For the fixed-point problems \eqref{eq: problem}, a direct calculation has been proved to be linearly convergent \cites{lin2013elliptic,cances2021convergence}, suggesting that the number of iterations is $\Theta(\log \frac{1}{\epsilon} ).$ To obtain the corresponding complexity of a stochastic algorithm,  
 we use the same techniques  for the previous theorems, and deduce the following inequality for the linear mixing scheme \eqref{alg: linear},
\begin{equation}
\mathbb{P}\left\{\sup_{m+1\leq n\leq j}\|e_n\|_2>\rho|\{q_i\}_{i=1}^m\right\}\mathbb{I}_{\{\{q_i\}_{i=1}^m\subset B(q^*,\rho)\}}\leq\frac{V_1(X_1)}{\rho^2},
\end{equation}where  $V_1$ is defined as,
\begin{equation}\label{vx1}
    V_1(X_1)=\|X_1\|_1+\Xi\sum_{i=1}^j \chi_i.
\end{equation}

We denote by $E_j:=\{q_n\in B(q^*,\rho),\;\forall n\in \{1,2,..,j\}\}$, i.e., the event that the first $j$ iterates lie in $B(q^*,\rho)$.

\begin{theorem}With the mixing coefficient in \eqref{mixingcoeff} and a non-negative sequence $\{a_n\}$ with $a_n\leq\frac{1-c}{(1+c)^2}$, \label{concentlinear} the linear mixing scheme \eqref{alg: linear} with initial guess $\{q_i\}_{i=1}^m\subset B(q^*,\rho)$ satisfies that $\mathbb{P}\{E_j|q_1\}\geq 1-\frac{V_1(X_1)}{\rho^2}$ and additionally, 
\begin{equation}
    \mathbb{E}\big[\|\bar{q}-q^*\|_2^2\mathbb{I}_{E_j}|q_1\big]\leq u_j,
\end{equation}where $\bar{q}$ is an averaged solution,  
\begin{equation*}
\begin{split}
    &\bar{q}:=\sum_{n=1}^j\left(\frac{A_n}{{\sum_{n=1}^jA_n}}\right)q_n,\\
    &u_j:=\frac{V_1(X_1)}{(1-c)\left(\sum_{n=1}^jA_n\right)},
\end{split}
\end{equation*}where $A_n$ is defined as $A_n:=\sum_{j=1}^mb_{m-j+1}a_{n+m-2+j}$    and $V_1(X_1)$ is from  \eqref{vx1}.

Consequently, the iterations from the linear mixing scheme $(m\geq 2)$ follow the inequality
\begin{equation}\label{concentineq}
    \mathbb{P}\Big\{\|\bar{q}-q^*\|_2\leq \epsilon|q_1\Big\}\geq 1-\frac{V_1(X_1)}{\rho^2}-\frac{u_j}{\epsilon^2}.
\end{equation}
\end{theorem}
We refer the readers to the proof in \ref{proofconcent}.

\medskip

Now, to obtain a specific complexity estimate, we consider the damping parameter $a_n=\frac{a}{n^{\beta}}$ for $\beta\in(\frac{1}{2},1)$ and the mixing coefficient in \eqref{mixingcoeff}. By applying the integral test to $a_n=\frac{a}{n^{\beta}}$, one can bound $V_1(X_1)$ as
\begin{equation}
    V_1(X_1)\leq \|X_1\|_1+\Xi\sum_{n=1}^{\infty}\chi_n=:C_{a,\beta}<\infty.
\end{equation}

To proceed, we call $q$ an $\epsilon$-solution if $\|q-q^*\|_2\leq\epsilon$. By applying the above theorem, we obtain a corollary as follows. 
\begin{col}\label{col: itercomplexity}
With the mixing coefficient in \eqref{mixingcoeff}, suppose that the damping parameter is given by $a_n=\frac{a}{n^{\beta}}$ with $\beta\in(\frac{1}{2},1)$ and sufficiently small $a>0$. Then, for any tolerance $\epsilon>0$ and failure probability $\gamma\in\left(\frac{C_{a,\beta}}{\rho^2},1\right)$, with probability at least $1-\gamma$, the linear mixing scheme finds an $\epsilon$-solution within the number of iterations,
\begin{equation}
    j=\Theta\left(\left(\frac{1}{\epsilon^2\left(\gamma-\frac{C_{a,\beta}}{\rho^2}\right)}\right)^{\frac{1}{1-\beta}}\right).
\end{equation}
$\Theta$ here is the notation for complexity. 
\end{col}
\begin{proof}
By letting $\gamma=\frac{V_1(X_1)}{\rho^2}+\frac{u_j}{\epsilon^2}$ in \eqref{concentineq}, using the upper bound $C_{a,\beta}$ for the first term, and observing that  $u_j=\mathcal{O}(j^{1-\beta})$ from the integral test, we deduce  this result for the averaged solution $\bar{q}$ from Theorem \ref{concentlinear}. 
\end{proof}

\section{A stochastic framework for a tight-binding approximation of DFT}\label{section4}

\subsection{The DFTB+ model}
As a discretization of Density Functional Theory (DFT), tight-binding (TB) approaches have  been  widely applied for larger electronic systems, particularly due to the fact that they do not require meshes. Among various TB schemes, SCC-DFTB \cite{elstner1998self} has shown great success for many different molecular and material systems. Part of the success can be attributed to  the incorporation of long-range Coulomb interactions. In addition, the implementation allows a self-consistent calculation to determine the charge distribution. Here, we briefly introduce SCC-DFTB \cite{elstner1998self}.
     
We let $M$ and $N$,  $M \geq N$, be respectively the number of atomic orbitals and nuclei. The atomic orbitals can be naturally labelled by $[M]=\{1,2,\cdots, M\}$. Let $\alpha^j$ be a multi-index for the atomic orbitals assigned to the $j$-th atom, i.e., $\alpha^j=\{\alpha^j_1, \alpha^j_2, \cdots, \alpha^j_{m_j}\} \subset [M]$; $\sum_{j=1}^N m_j=M$.  For instance, by $\nu\in\alpha^j$ we mean that the atomic orbital $\nu$ is associated with the $j$-th atom.  Such notations are particularly useful for a system with multiple species, for which $m_j$ varies.  Further, they can be used to indicate those elements in the Hamiltonian matrix $H$ and the overlap matrix $S$ that represent interactions among the atoms, as we explain next. We denote the list of electronic charges associated with the atoms by $q=(q(1),..,q(N))^T\in\mathbb{R}^{N}$.

The DFTB+ model involves a generalize eigenvalue problem and Hamiltonian corrections using linear response. In particular, the algorithm in   \cite{elstner1998self} finds a solution of charge vector $q$ by iterating through the following equations,
  \begin{equation}\label{sccdftb}
        \begin{aligned}
                   &Hc_i=\epsilon_iSc_i \textrm{ with the eigenpair }(\epsilon_i,c_i), \\
            &q(j)=\frac{1}{2}\sum_{i=1}^{M}n_i\sum_{\mu\in\alpha^j}\sum_{\nu=1}^{M}(c_{\mu i}^*c_{\nu i}S_{\mu\nu}+c_{\nu i}^*c_{\mu i}S_{\nu\mu}), \; j=1, 2, \cdots, N \\
            &H_{\mu\nu}^1=\frac{1}{2}S_{\mu\nu}\sum_{j=1}^N(\gamma_{ij}+\gamma_{kj})\Delta q(j),\\
            &H=H^{0}+H^{1},\\
        \end{aligned}
    \end{equation}
    where 
\begin{equation}\label{eq: HSn}
       n_i=f(\epsilon_i), \quad 
       H_{\mu\nu}^0=\langle \varphi_{\mu}|\hat{H}_0|\varphi_{\nu}\rangle,\quad 
       S_{\mu\nu}=\langle\varphi_{\mu}|\varphi_{\nu} \rangle.
\end{equation}    
The function $f$ denotes the occupation numbers of electrons. As an example, one can consider the Fermi-Dirac distribution:
\begin{equation}\label{eq: fermi-dirac}
    f(x)=\frac{2}{1+\exp(\beta(x-\mu))},
\end{equation} 
with $\mu$ being the Fermi energy and $\beta$ being the inverse temperature. Specifically, $n_i=f(\epsilon_i)$ denotes the occupation number for the energy level $\epsilon_i$.

To explain the notations, here we briefly outline the algorithm in the DFTB model.
In the implementation,  one starts with a set of preselected localized atomic orbitals $\{\varphi_{\mu}\}$, the symmetric matrices $H^0\in \mathbb{R}^{M\times M}$ and $0\prec S\in \mathbb{R}^{M\times M}$ in \eqref{eq: HSn} are defined as the Hamiltonian matrix with the non-SCC TB method \cite{elstner1998self} and the usual overlap matrix, respectively. In the SCC-DFTB procedure, they are parameterized in terms of the nuclei positions. The first line of \eqref{sccdftb} amounts to a diagonalization of the pair $(H,S)$, with $c_{\nu i}$ denoting the $\nu$-th element of the eigenvector $c_i$ (which corresponds to the atomic orbital $\nu$).
The eigenvalues and eigenvectors are then used to compute the electronic charges $q$. With the updated charges, one updates the matrix $H^1$ and the total Hamiltonian before the algorithm enter the next iteration.

  In $H_1$,  $\mu$ and $\nu$ are labels of two atomic orbitals. 
Since there might be multiple species in the system, they are designated as multi-indices including the orbitals associated with the $i$-th atom and the $k$-th atom, respectively. 
In addition, the coefficient $\gamma_{ij}$ accounts for the Coulomb interaction between the $i$-atom and the $j$-th atom. In addition, the charge fluctuation $\Delta q(j)$ is defined as  $q^0(j)-q(j)$, where $q^0(j)$ is the electronic charge when the atom is in isolation. The formal description on the role of both the quantities $\gamma_{ij}$ and $\Delta q(j)$ in the DFTB+ model \eqref{sccdftb} is beyond the focus of this paper. For more details, we refer readers to \cite{elstner1998self}.

The steps in \eqref{sccdftb} can be repeated until the charge vector, $q$, converges. The SCF problem can be reduced to a fixed-point iteration problem  (FPI) \cite{lin2013elliptic}, which for the SCC-DFTB model, can be described as follows.  Given $q_n$ as the input, we update the Hamiltonian $H_{n}=H^0+H^1(q_{n})$ and solve the generalized eigenvalue problem, $H_nU_n=SU_n\Lambda_n$. Using the eigenvalues and eigenvectors, we compute $q_{n+1}$ as the output according to the first equation in \eqref{sccdftb}. This procedure can be simplified to a fixed-point iteration, 
\begin{equation}\label{fpiscc}
     q_{n+1}=K(q_n). 
\end{equation} 
The mapping $K$ will be expressed as a matrix-vector form \eqref{mapK} as we will demonstrate in the next subsection.
 
After obtaining an approximate limit, the force $F=(F_{\alpha})\in\mathbb{R}^N$ can be computed from the total energy $E$, 
\begin{equation}
   F_{\alpha}=-\frac{\partial E}{\partial R_{\alpha}}, \quad 
E:=\sum_{i=1}^Mn_i\epsilon_i+E_{rep},
\end{equation}
where $R=(R_{\alpha})\in\mathbb{R}^N$ and $E_{rep}$  describes the repulsion between the nuclei.  The calculation of the forces enables geometric optimizations and molecular dynamics simulations \cite{elstner1998self}. In this paper, we will only focus on the charge iterations. The integration with the force calculation will be addressed in separate works.

\medskip
 
 A direct implementation of \eqref{fpiscc}, however, usually does not lead to a convergent sequence, mainly due to the lack of contractiveness of the mapping $K$. Practical computations  based on \eqref{fpiscc} are often accompanied with a mixing and damping strategy as discussed in \cites{fang2009two,lin2013elliptic,toth_convergence_2015}. For example, one can use the simple mixing scheme \eqref{eq: simplem}. 
 More generally, mixing methods  \cites{fang2009two,alber2012stochastic}, which use multiple previous steps, such as the linear mixing \eqref{eq: linearm},  are commonly employed in practice.

\subsection{Matrix representation for charge functions}In this section, we present an expression of the charge at an atom in terms of the trace of a matrix.  This is an important step towards the construction of stochastic algorithms. A close inspection of the coefficients in the first line of the equation \eqref{sccdftb} reveals the following formula. 
\begin{lemma}\label{charge}The  electronic
charge associated with the $j$-th atom,  $q(j)$ in the system \eqref{sccdftb}, can be expressed in terms of the trace of a matrix as
\begin{equation}\label{eq: q-trace}
   q(j)=\mathrm{tr}\left(E_{j}^TLf(A)L^{-1}E_{j}\right),
\end{equation}where $A=L^{-1}HL^{-T}$ with the Cholesky factorization $S=LL^T$. Here  $E_{j}\in \mathbb{R}^{M\times m_j}$ is the rectangular submatrix of the $M\times M$ identity matrix obtained by pulling out the columns according to the multi-index $\alpha^j$ (with dimension denoted by $m_j$).
\end{lemma}
\begin{proof}Denote by $S_{j}$ the rectangular submatrix of the overlap matrix $S$,  with columns associated with  indices in $\alpha^j$. Define $I_{j}:=E_jE_j^T\in\mathbb{R}^{M\times M}$. We use the spectral decomposition 
\begin{equation}
(L^T)^{-1}f(A)L^{-1}=\sum_{i=1}^Mf(\epsilon_i)c_ic_i^T,
\end{equation}where $(c_i,\epsilon_i)$ is the eigenpair defined in \eqref{sccdftb}.

First, we can rewrite the first equation in \eqref{sccdftb} as follows
\begin{equation*}
    q(j)=\frac{1}{2}\sum_{i=1}^{M}n_i\sum_{\mu\in\alpha^j}\sum_{\nu=1}^M(c_{\mu i}^Tc_{\nu i}S_{\mu\nu}+c_{\nu i}^Tc_{\mu i}S_{\nu\mu})=\frac{1}{2}\sum_{i=1}^{M}n_i(c_i^TI_jSc_i+c_i^TSI_jc_i).
\end{equation*}By using the commutative property of the trace, the first term in the summand can be rewritten as follows
\begin{equation*}
    c_i^TI_{j}Sc_i=\mathrm{tr}(c_i^TI_{j}Sc_i)=\mathrm{tr}(c_ic_i^TI_{j}S)=\mathrm{tr}(c_ic_i^TE_{j}E_{j}^TS)=\mathrm{tr}(c_ic_i^TE_{j}S_{j}^T).
\end{equation*}

By a similar treatment of the second summand, we can rewrite the charge $q(j)$ 
\begin{equation*}
    \begin{aligned}
         q(j)&=\frac{1}{2}\sum_{i=1}^{M}n_i \mathrm{tr}(c_ic_i^T(E_{j}S_{j}^T+S_{j}E_{j}^T))\\
          &=\frac{1}{2}\mathrm{tr}(L^{-T}f(A)L^{-1}(E_{j}S_{j}^T+S_{j}E_{j}^T))\\
          &=\mathrm{tr}(E_{j}^T(L^T)^{-1}f(A)L^{-1}S_{j})=\mathrm{tr}(E_{j}^TL^{-T}f(A)L^{T}E_{j})=\textrm{tr}(E_{j}^TLf(A)L^{-1}E_{j}).
    \end{aligned}
\end{equation*}The second equality holds by the spectral decomposition  shown  above. In the last line,  the identity $L^{-1}S_{j}=L^TE_{j}$ is used.
\end{proof}

We now turn to the third equation in  \eqref{sccdftb}, which updates the Hamiltonian matrix at each iteration in the SCC-DFTB procedure. The equation is given in terms of the Hamiltonian $H$ and the overlap matrix $S$. However, as shown in the preceding lemma, the matrix of the specific form $L^{-1}HL^{-T}$ is required to update the charge $q_{\alpha}$. This leads to reformulation of the third equation in \eqref{sccdftb}. Here, we  introduce notations as follows: The symbol $\mathrm{sym}$ stands for the symmetrization $\mathrm{sym}(A)=A+A^T$. In addition, we define  $e\otimes_{N}v$ with $v=(v_1,v_2,..,v_N)^T$  as follows,
\begin{equation*}
    e\otimes_{N}v:=(\underset{m_1}{\underbrace{v_1,\dots,v_1}},\dots,\underset{m_N}{\underbrace{v_N,\dots,v_N}})^T,
\end{equation*}
where $m_j$,  the number of atomic orbitals associated with the $j$-th atom, indicates  the number of times the element $v_j$ is repeated. As opposed to the Kronecker product notation $\otimes$, this operation copies each element of the vector $v$ as many times as the corresponding index $\alpha^j$.
\begin{lemma}
The third equation in \eqref{sccdftb} has the following alternative expression,
\begin{equation}\label{A}
    A=A_0+\frac{1}{2}\mathrm{sym}\big(L^{-1} \mathrm{diag}(e\otimes_{N} \Gamma \Delta q)L\big),
\end{equation} where $A_0=L^{-1}H_0L^{-T}$ and $\Gamma:=(\gamma_{ij})\in\mathbb{R}^{N\times N}$.
\end{lemma}

We prove this lemma as follows.
\begin{proof}
In the third line of \eqref{sccdftb}, the correction term can be rewritten as
\begin{equation*}
    \begin{aligned}
         &H_{\mu\nu}^1=\frac{1}{2}S_{\mu\nu}\sum_{j=1}^N(\gamma_{ij}+\gamma_{kj})\Delta q(j),\quad \Delta q:=(\Delta q(1),...,\Delta q(N))^T\\
         &=\frac{1}{2}S_{\mu\nu}\bigg(\Gamma_i^T\Delta q+\Gamma_k^T\Delta q\bigg),\quad \Gamma_{i}\textrm{ is the }i\textrm{-th row vector of }\Gamma\\
         &=\frac{1}{2}S_{\mu\nu}\bigg(i-\textrm{th entry of }\Gamma\Delta q+k-\textrm{th entry of }\Gamma\Delta q\bigg)
    \end{aligned}
\end{equation*}
        \begin{equation*}
        \Longrightarrow H^1=\frac{1}{2}(\underset{\textrm{row operation by }\mu}{\underbrace{\mathrm{diag}(e\otimes_{N}\Gamma\Delta q)S}}+\underset{\textrm{column operation by }\nu}{\underbrace{S\mathrm{diag}(e\otimes_{N}\Gamma\Delta q)}}).
    \end{equation*}
The desired expression is obtained by multiplying $L^{-1}$ to left and $L^{-T}$ to right.  
\end{proof}

In summary, the system \eqref{sccdftb} can be concisely rewritten as
\begin{equation}\label{finalsystem}
\begin{split}
    & q(j)=\textrm{tr}(E_{j}^TLf(A)L^{-1}E_{j}),\\
    & A=A_0+\frac{1}{2}\text{sym}(L^{-1}\text{diag}(e\otimes_{N} \Gamma \Delta q)L).\\
\end{split}
\end{equation} We note that the generalized eigenvalue problem in the system \eqref{sccdftb} is incorporated in the system \eqref{finalsystem} implicitly. Especially based on the charge $q(j)$ in the system \ref{finalsystem}, the mapping $K$ in \eqref{fpiscc} can be explicitly formulated as follows,
\begin{equation}\label{mapK}
 K: \mathbb{R}^N \to  \mathbb{R}^N, \quad   K(q) =\begin{bmatrix}
    \mathrm{tr}(E_{1}^TLf(A)L^{-1}E_{1})\\ \mathrm{tr}(E_{2}^TLf(A)L^{-1}E_{2}) \\
    \vdots\\
    \mathrm{tr}(E_{N}^TLf(A)L^{-1}E_{N}) 
    \end{bmatrix},
\end{equation}
where $N$ denotes the number of nuclei. We recall that the matrix $A$ in the right-hand side of \eqref{finalsystem} involves the charge vector $q$ within the term $\Delta q$, which means that $K$ is a mapping of $q$.

\subsection{A stochastic algorithm for the DFTB+ model}  According to \eqref{mapK}, one can directly calculate $K(q)$ when the diagonal of the matrix $Lf(A)L^{-1}$ is explicitly known, while it requires the eigen decomposition of $A$, which is expensive for large matrices. Alternatively, we employ the diagonal estimator \ref{Diagonal} (see similar applications to electronic structure calculations \cites{thicke2019accelerating,bekas2007estimator}). Within this diagonal estimator, one has to compute a matrix-vector product, in our case, $f(A)L^{-1}v$, which requires the eigen decomposition of $A$ again. To bypass a full diagonalization, we use the Krylov subspace method  \cites{saad1992analysis,diele2009error} to approximate the matrix-vector product. This method yields a fairly good approximation for sparse matrices. Error estimates of the Krylov approximation have been proposed in \cites{diele2009error,saad1992analysis,eiermann2006restarted} for the case of the exponential-like functions.  However, the Fermi-Dirac distribution \eqref{eq: fermi-dirac} clearly does not belong to this family of functions, and an error estimate requires a different proof.
\begin{theorem}\label{krylovapprox}
Suppose that $A$ is a symmetric matrix and $f(x)$ is the Fermi-Dirac distribution in \eqref{eq: fermi-dirac}. Then, for any integer $\ell>s$, the error of the Krylov subspace method can be bounded by,
 \begin{equation}\label{eq: err-subs-approx}
\big\| f(A)v -  \|v\|_2 V_{\ell}f(T_{\ell})e_1 \big\|_2\leq \frac{4M(\rho)\|v\|_2\rho^{-\ell}}{\rho-1},
 \end{equation}where $V$ is the total variation of $f^{(s)}(x)$ and the constants $M(\rho)$ and $\rho>1$ depend only on $f(x)$. Consequently, as the degree $\ell$ increases, one expects the accuracy from the Krylov approximation to improve, namely,
 \begin{equation*}
     \lim_{\ell\to\infty}\|v\|_2 V_{\ell}f(T_{\ell})e_1=f(A)v.
 \end{equation*}
\end{theorem}
The proof of this theorem, using tools from spectral approximations in the previous works \cites{trefethen2019approximation,xi2018fast}, is given in Appendix \ref{sec: A1}.  
\begin{remark}
   The theorem still holds for any continuously differentiable function that can be extended analytically to some Bernstein ellipse according to results in \cite{trefethen2019approximation}. Furthermore, this shows that the Krylov subspace approximation with such a function improves error bound for the Chebyshev approximation by noticing the inequality \eqref{chebyapprox}. To be specific, the error bound with Chebyshev approximation decays in a polynomial order of $\ell$, but the Krylov subspace approximation has an exponential decay in the error bound.
 \end{remark}

Now, by combining the diagonal estimator \eqref{Diagonal} and the Krylov subpspace approximation \eqref{krylovapprox}, we estimate the diagonal of the matrix $Lf(A)L^{-1}$ as follows, 
\begin{equation}\label{approximation}
\mathrm{diag}(Lf(A)L^{-1})=\mathrm{diag}(\mathbb{E}[ Lf(A)L^{-1}vv^T])\approx \mathrm{diag}(\mathbb{E}[\|L^{-1}v\|_2 LV_{\ell}f(T_{\ell})e_1v^T]),
\end{equation}where $V_{\ell}\in\mathbb{R}^{M\times \ell}$ is the left-orthogonal matrix, $T_{\ell}\in\mathbb{R}^{\ell\times \ell}$ is the tridiagonal matrix from the Lanczos method of $\ell$ steps and $v$ is a random vector whose covariance is the identity matrix.  Here, $e_1$ is the first standard basis vector in $\mathbb{R}^{\ell}$.

Especially, for each $\ell$, the approximation \eqref{approximation} is reduced to the relation
\begin{equation}\label{Kk}
 K(q)\approx K_{\ell}(q) = \mathbb{E}[k_{\ell}(q,v)],
\end{equation}
where the expectation is taken over the random vector $v$ and the random mapping $k_{\ell}(q,v)$ is defined similar to \eqref{mapK},
\begin{equation}\label{sampling}
    k_{\ell}(q,v)=\|L^{-1}v\|_2 
     \left[\begin{array}{c}
    \mathrm{tr}(E_{1}^TLV_{\ell}f(T_{\ell})e_1v^T  E_{1})\\ 
    \mathrm{tr}(E_{2}^TLV_{\ell}f(T_{\ell})e_1v^T E_{2}) \\
    \vdots\\
    \mathrm{tr}(E_{N}^TLV_{\ell}f(T_{\ell})e_1v^T E_{N}) \end{array}\right].
\end{equation}We note that the average $K_{\ell}(q)$ is not equal to the original mapping $K(q)$ because of the error from the subspace approximation method \eqref{eq: err-subs-approx}. In other words, this yields an approximate fixed-point problem.  Nevertheless, we can make the approximation error negligible by selecting a sufficiently large $\ell$, as shown in Theorem \ref{krylovapprox}.

\begin{algorithm}
\SetAlgoLined
	\KwData{$A,L,\ell,n_{vec}$}
	\KwResult{$\frac{1}{n_{vec}}\sum_{i=1}^{n_{vec}}k_{\ell}(q,v_i)$, Approximation of $K(q)$ }
	 Samples $v_1,v_2,...,v_{n_{vec}}$
    
     Define $V_1=[v_1,v_2,...,v_{n_{vec}}]$
     
	 $V_2=L^{-1}V_1$
	 
	\For{$i=1:n_{vec}$}{
		     $v=V_2(:,i)$\;
		    $[\|v\|_2,V_{\ell},T_{\ell}]=\textrm{Lanczos}(A,v,\ell)$\;
		    $V_2(:,i)=\|v\|_2 V_{\ell}f(T_{\ell})e_1$, \textrm{Krylov subspace approximation for $f(A)v$}\;}
	 $V_2=LV_2$, Approximation of the matrix $Lf(A)L^{-1}V_1$

	 Compute the average $\frac{1}{n_{vec}}\sum_{i=1}^{n_{vec}}\mathrm{diag}(V_2(:,i)V_1(:,i)^T)\approx \textrm{diag}(Lf(A)L^{-1})$
	 
	 Compute the Monte-Carlo sum,  $\frac{1}{n_{vec}}\sum_{i=1}^{n_{vec}}k_{\ell}(q,v_i)$ using \eqref{sampling}
    \caption{Stochastic Lanczos method for the charge function in \eqref{mapK}}
    \label{alg: StoLan}
\end{algorithm}

After obtaining the output $[\|v\|_2,V_{\ell},T_{\ell}]$ from the Lanczos method, an eigensolver should be implemented for the eigen decompostion of $T_{\ell}$ in order to perform the Krylov subspace approximation with $f(T_{\ell})=U_{\ell}f(D_{\ell})U_{\ell}^T$. As compared to the original system, this is a much smaller matrix and the computation is much easier. 

Finally, by incorporating Algorithms \ref{alg: linear} and \ref{alg: StoLan} into the system \eqref{finalsystem}, we arrive at a stochastic self-consistent algorithm for the DFTB+ model outlined in {\bf Algorithm} \ref{alg: stodftb}.
\begin{algorithm}
\SetAlgoLined
	\KwData{$\{a_n\}$, $\{b_i\}_{i=1}^m$,$\{q_i\}_{i=1}^m$, $A_0=L^{-1}H_0L^{-T}$, $\ell$, $n_{vec}$}
	\KwResult{Approximate fixed-point}
	\For{$n=m,m+1,...,\;$until convergence}{
    $q_{n+1}=(1-a_n)B_m(q_{n})+a_n B_m(k(q_n,v_n))$\;
    $A_{n+1}=A_0+\frac{1}{2}\text{sym}(L^{-1}\text{diag}(e\otimes_{N} \Gamma \Delta q_{n+1})L)$\;
    $k(q_{n+1},v_{n+1})=\textrm{StoLan}(A_{n+1},L,\ell,n_{vec})$, implement Algorithm \ref{alg: StoLan} \;}
    	\caption{Stochastic Self-Consistent Calculation for the DFTB}
    	\label{alg: stodftb}
\end{algorithm}

\begin{remark}
In the computation of the stochastic function $k(q_n,v_n)$, we have assumed that the Fermi level $\mu$ is given. In the stochastic algorithm framework, this can be done very efficiently using the trace estimator \cites{xi2018fast,lin2016approximating} for the density of states, which can be subsequently used to estimate the Fermi level.
\end{remark}

We denote the stochastic noise from the relation \eqref{Kk} by
\begin{equation}\label{xi}
 \xi_{\ell}(q,v)=k_{\ell}(q,v)-K_{\ell}(q).   
\end{equation}By applying Theorem \ref{krylovapprox}, we obtain the following result,
\begin{col}\label{noise}
The stochastic noise \eqref{xi} has zero mean and a bounded variance in the ball $\mathcal{B}(q^*,\rho)$.
\end{col}
\begin{proof}
This is because $K(q)$ is continuous and Theorem \ref{krylovapprox} guarantees the boundedness of the approximation \eqref{approximation}. In addition, by the relation \eqref{Kk}, the stochastic noise has zero mean.
\end{proof}

\begin{remark}
   In sharp contrast to the deterministic counterpart \eqref{eq: simplem}, the term $k(q_n,v_n)$ in \eqref{alg: stodftb} emphasizes the point that the quantity is only sampled once, motivated by the remarkable success of the stochastic algorithms \cite{robbins1951stochastic}. This leads to a significant reduction of one iteration cost and a potential application of the stochastic framework for large-scale systems. 
\end{remark}

\subsection{A Preliminary Comparison of Stochastic and Direct SCC-DFTB}
In the implementation of stochastic approximation methods, three sources of error arise: {\it approximation, estimation} and {\it optimization} \cite{bottou2007tradeoffs}. In our case, due to the diagonal estimator \eqref{Diagonal} in which the true distribution of the random vector $v$ is determined by users, we do not consider an estimation error. Rather, we focus on the approximation error and the optimization error. 

To be precise, we aim at estimating how close some iterate $q$ obtained from a stochastic algorithm is to a solution $q^*$,
\begin{equation}\label{twoerrs}
    \|q-q^*\|_2\leq \underset{\textrm{optimization error}}{\|q-q_{\ell}^*\|_2}+\underset{\textrm{approximation error}}{\|q_{\ell}^*-q^*\|_2},
\end{equation}where $q_{\ell}^*$ and $q^*$ are fixed points of the mapping $K_{\ell}$ in \eqref{Kk} and the exact mapping $K$ in \eqref{mapK}, respectively. 

To quantify the approximation error, we use Theorem \ref{krylovapprox}. We assume that $K$ \eqref{mapK} satisfies the stability condition in Theorem \ref{thmcontract'} and the same for $K_{\ell}$ \eqref{Kk} for sufficiently large $\ell$ based on Theorem \ref{krylovapprox}. As we discussed in Theorem \ref{thmcontract'}, these mappings can be transformed to contractions with a small auxiliary parameter $\theta>0$. Therefore, without loss of generality, we assume that $K$ and $K_{\ell}$ satisfy the contractiveness {\bf Assumption} \ref{as:04} throughout the following analysis.

First, we derive an bound for the approximation error. By definitions of the solutions $q^*$ and $q_{\ell}^*$, we have
\begin{equation*}
    \begin{split}
        q_{\ell}^*-q^*=K_{\ell}(q_{\ell}^*)-K_{\ell}(q^*)+K_{\ell}(q^*)-K(q^*).
    \end{split}
\end{equation*}By using the triangle inequality, one has,
\begin{equation*}
    \|q_{\ell}^*-q^*\|_2\leq \|K_{\ell}(q_{\ell}^*)-K_{\ell}(q^*)\|_2+\|K_{\ell}(q^*)-K(q^*)\|_2\leq c_{\ell}\|q_{\ell}^*-q^*\|_2+C(\rho)\rho^{-\ell},
\end{equation*}where $c_{\ell}\in(0,1)$ is associated with the contraction $K_{\ell}$. To derive the last term $C(\rho)\rho^{-\ell}$, we recall the relation \eqref{approximation} with $K(q)$ in \eqref{mapK} and $K_{\ell}(q)$ in \eqref{Kk}. In the case of the Hutchinson estimator \eqref{Diagonal} where the sample space consists of a finite number of random vectors of the same length,  we apply Theorem \ref{krylovapprox} to  \eqref{approximation},
\begin{equation*}
\begin{split}
    &\|\mathbb{E}[ Lf(A)L^{-1}vv^T]- \mathbb{E}[\|L^{-1}v\|_2 LV_{\ell}f(T_{\ell})e_1v^T]\|_2\\
    &\leq \|\mathbb{E}[ L\left(f(A)L^{-1}v-\|L^{-1}v\|_2V_{\ell}f(T_{\ell})e_1\right)v^T]\|_2\\
    &\leq \mathbb{E}[\| L\|_2\cdot\|f(A)L^{-1}v-\|L^{-1}v\|_2V_{\ell}f(T_{\ell})e_1\|_2\cdot\|v^T\|_2]\\
    &\leq \|L\|_2\|v\|_2\frac{4M(\rho)\|L^{-1}v\|_2\rho^{-\ell}}{\rho-1}.
\end{split}
\end{equation*}This can be used to estimate the difference $\|K_{\ell}(q^*)-K(q^*)\|_2$ by noticing the definitions $K(q)$ in \eqref{mapK} and $K_{\ell}(q)$ in \eqref{Kk}. Overall, the approximation error is estimated as
\begin{equation}\label{approxerr}
    \|q_{\ell}^*-q^*\|_2\leq \frac{C(\rho)\rho^{-\ell}}{1-c_{\ell}}.
\end{equation}
 
Now we turn to the optimization error based on the result in Section \ref{cmp}. In particular,  this can be regarded as a route to compare the stochastic and direct methods. In computing the Mulliken charge $q(j)$ in the system \ref{sccdftb}, the computation involved in the eigenvalue problem scales $\mathcal{O}(M^3)$. In contrast, this only scales  $\mathcal{O}(\ell M^2)$ within the stochastic Lanczos method \ref{alg: StoLan}. The cost for the rest of the procedure in both methods is negligible due to the sparsity of the matrices. This observation leads to the comparison in Table \eqref{tab}.

\begin{table}
\begin{tabular}{||c c c c ||}
 \hline
 Algorithm & Cost of one iteration & Iterations to reach $\epsilon$ & Total cost \\ [0.5ex] 
 \hline
 Stochastic &  $\mathcal{O}(\ell M^2)$ & $\mathcal{O}\left(\epsilon^{-\frac{2}{1-\beta}}\right)$ &  $\mathcal{O}\left(\ell M^2\epsilon^{-\frac{2}{1-\beta}}\right)$ \\
 \hline
 Exact & $\mathcal{O}(M^3)$ & $\Theta\left(\log\frac{1}{\epsilon}\right)$ & $\mathcal{O}\left(M^3\log\frac{1}{\epsilon}\right)$ \\
 \hline
\end{tabular}
\caption{Comparison between stochastic and direct algorithms \eqref{fpiscc} in terms of optimization error with tolerance $\epsilon$. Here $M$ is the dimension of the Hamiltonian or overlap matrix. \label{tab}}
\end{table}

We are now in a better position to compare the direct  method \eqref{fpiscc} and the stochastic method \eqref{alg: stodftb}. To reach accuracy $\epsilon$ in \eqref{twoerrs}, Table \eqref{tab} implies that the direct method requires the following time
\begin{equation*}
    \mathcal{O}\left(M^3\log\frac{1}{\epsilon}\right),
\end{equation*}whereas
the stochastic method requires the time, for $\beta\in(\frac{1}{2},1)$ in \ref{col: itercomplexity},
\begin{equation*}
    \mathcal{O}\left(\ell M^2\widetilde{\epsilon}^{-\frac{2}{1-\beta}}\right),
\end{equation*}where
\begin{equation*}
    \widetilde{\epsilon}:=\epsilon-\frac{C(\rho)\rho^{-\ell}}{1-c_{\ell}}
\end{equation*}by using the estimate \eqref{approxerr}.

Therefore, the stochastic method could become advantageous when the number of orbitals, $M$, is larger than
\begin{equation}
    M=\mathcal{O}\left(\ell\left(\widetilde{\epsilon}^{\frac{2}{1-\beta}}\log\frac{1}{\epsilon}\right)^{-1}\right),
\end{equation}This was motivated by the previous work \cite{bottou2007tradeoffs} to compare stochastic optimization methods with direct counterparts. 

\begin{remark}
   A closer inspection of the Krylov subspace approximation \eqref{krylovapprox} implies that the temperature influences the quality of this approximation. Especially, at a low temperature $T\ll 1$, as we discuss after Theorem \ref{thmb4}, the minor axis of a Bernstein ellipse should be small enough to guarantee a reasonable approximation. More specifically, by the definition of the Bernstein ellipse \cite{trefethen2019approximation}, an optimal radius $\rho$ can be found from the following equation
   \begin{equation}\label{radius}
       \rho-\frac{1}{\rho}= \frac{4\pi k_BT}{\lambda_{\max}(A)-\lambda_{\min}(A)}=:c(T),
   \end{equation}which yields that when $T\ll 1$ 
   \begin{equation*}
       \rho = \frac{c(T)+\sqrt{c(T)^2+4}}{2}\approx 1+\frac{4c(T)+c(T)^2}{8}\approx 1+\frac{c(T)}{2}.
   \end{equation*}Therefore, at the low temperature, one can deduce that the approximation error of \eqref{twoerrs} decays exponentially with $\ell$ roughly as
   \begin{equation*}
       \left(1+\frac{2\pi k_BT}{\lambda_{\max}(A)-\lambda_{\min}(A)}\right)^{-\ell}.
   \end{equation*}Due to the factor $T$ in this expression, a sufficiently large $\ell$ should be chosen to obtain a reasonable approximation. In contrast, in the regime of high temperature, the Krylov approximation method becomes very effective, since a large value of $\rho$ can be selected from \eqref{radius}.
\end{remark}

In summary, the stochastic algorithm \eqref{alg: stodftb} offers a new framework for electronic structure calculations. One immediate question is when it is more efficient than a direct SCF method for specific application, e.g., biomolecules or crystalline solids. It is still a complex issue and a much more comprehensive study is needed to take into account many factors, e.g., parallelization, implementations with sparse matrix factorizations, choosing optimal mixing parameters, variance reduction techniques, etc. We leave these issues to future studies.

 \section{Numerical Results}\label{section5}

In this section, we present preliminary results from some numerical experiments. We consider a system of graphene with 800 atoms.  We have chosen the lattice spacing to be 1.4203 \AA. In the function $f$ \eqref{eq: fermi-dirac}, we set  the Fermi level to be -0.1648 and $\beta=1052.58,$ which corresponds to $300$ Kelvin. The Hamiltonian and overlap matrices, together with the matrix $\Gamma$ are all obtained from DFTB+ \cite{elstner1998self}. The dimension of these matrices is $3200 \times 3200.$

As a reference, the solution $q^*$ of \eqref{fpiscc} is first computed using the simple mixing scheme with damping $a=0.001$. In addition, upon convergence, we used a centered difference method with step size $h=0.001$,  and computed the Jacobian $K'(q^*)$. We found that all the eigenvalues are real and lie between $-6.8859$ and $-0.1348$. In light of Theorem \ref{thmcontract'}, the contraction assumption \eqref{eq: contractive} is fulfilled under some norm by choosing a small step size $a$. 

Since the original fixed point problem $q=K(q)$ has been replaced by $q=K_\ell(q)$, we first examine the error between the fixed points. Figure  \ref{fig:errvsd} shows how this error depends on the dimension $\ell$ of the subspace. The error here is measured in $\|\cdot \|_\infty$ norm and the norm of $q^*$ is around 4. One can observe that the error decreases when the subspace is expanded.  
 
\begin{figure}
    \centering
    \includegraphics[scale=0.2]{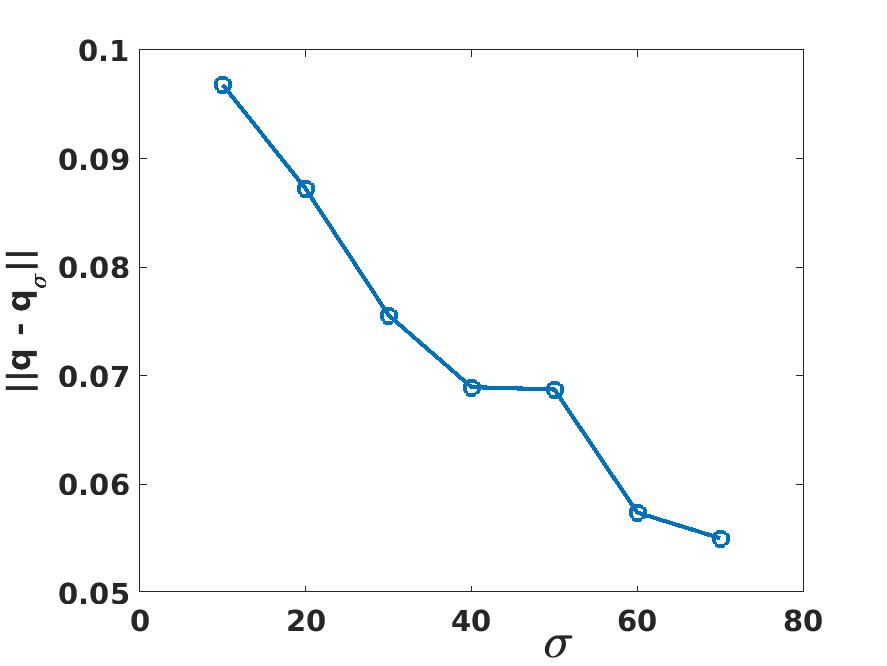}
    \caption{The error of the subspace approximation $K_\ell$: $q^*$ and $q_\ell^*$ are respectively the fixed-points of $K$ and $K_\ell$. The error is measured in the infinity norm. }
    \label{fig:errvsd}
\end{figure}

For the rest of the discussions, we choose $\ell=20$, and we regard the fixed point of $K_{20}(q)$ as the true solution $q^*$.

We test linear mixing methods (Algorithm \ref{alg: stodftb}). We pick uniform mixing parameters, i.e., { $b_i =1/m$}. In addition, we choose the damping parameter, $a_n= \min \{ \left(50 + 2n \right)^{-1}, 0.005\},$ which fulfills the conditions in the convergence theorem.  The error from 30,000 iterations are shown in Figure \ref{fig:linear}.   To mimic the mean error, we averaged the error over every 1,000 iterations. In addition, we run all the cases with simple mixing for 2,000 iterations, followed with the mixing schemes turned on, to allow these cases to follow the same initial period. Surprisingly, the linear mixing scheme does not seem to have faster convergence than the simple mixing.  To further test the convergence, we choose the damping parameter as follows,  $a_n=\min \{ [50 + 4 n^{3/4} ]^{-1}, 0.005\},$
and show the results in Figure \ref{fig:linear'}. Interestingly, with this choice of the damping parameter, using more mixing steps (larger $m$) yields faster convergence.

\begin{figure}
    \centering
    \includegraphics[scale=0.3]{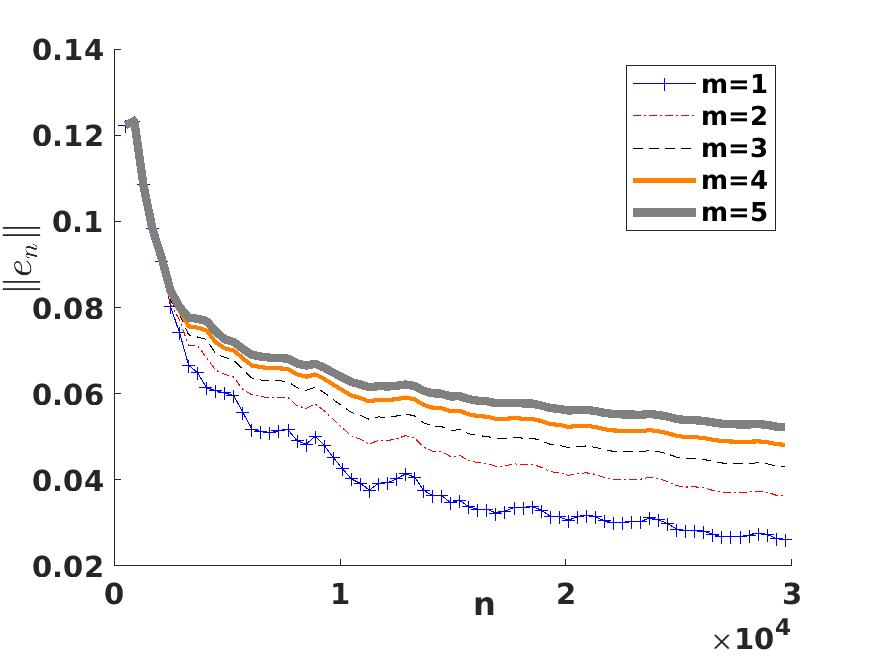}
    \caption{The error $\|\bar{q}_n-q^*\|_{\infty}$ from the linear mixing method (Algorithm \ref{alg: stodftb}) with $m=2,3,4, 5$ and $6$ using $a_n= \min \{ \left(50 + 2n \right)^{-1}, 0.005\}$. }
    \label{fig:linear}
\end{figure}
\begin{figure}
    \centering
    \includegraphics[scale=0.25]{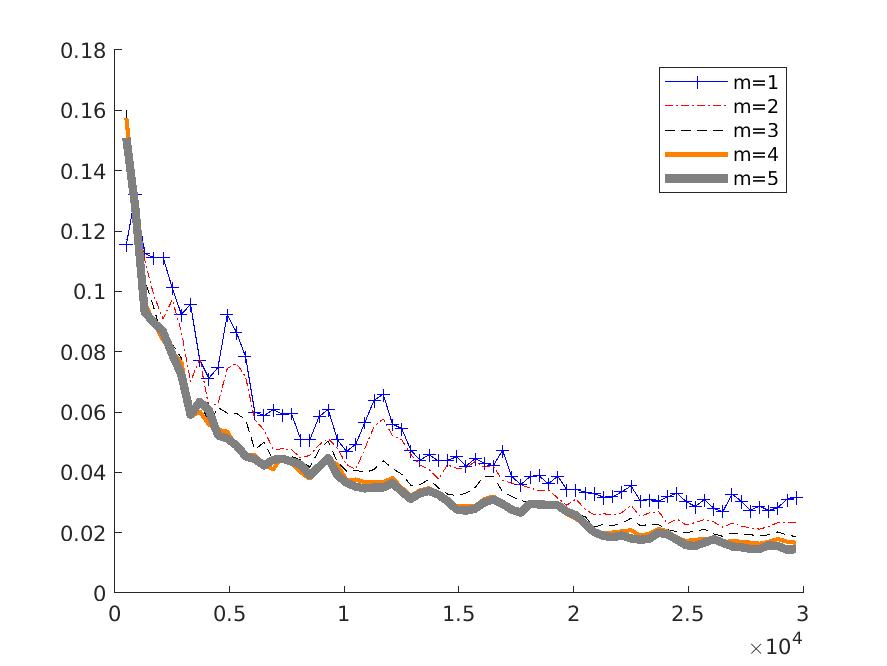}
    \caption{The error $\|\bar{q}_n-q^*\|_{\infty}$ from the linear mixing method (Algorithm \ref{alg: stodftb}) with $m=2,3,4, 5$ and $6$ using $a_n=\min \{ [50 + 4 n^{3/4} ]^{-1}, 0.005\}$. }
    \label{fig:linear'}
\end{figure}

 As a simple exposition, we applied the Anderson mixing method \cite{toth_convergence_2015} to the stochastic algorithm \eqref{alg: stodftb}. Note that the mixing coefficient $b_n$ is determined on the fly via a least squares procedure \cite{toth_convergence_2015}. Figure \ref{fig:anderson} displays the error from 30,000 iterations of the Anderson's method with $m=2,3,4$ and $5$. In the implementations, we choose $a_n= [50 + 4 n^{3/4} ]^{-1}.$  Again, due to the stochastic nature, we define the error to be 
\(\bar{q}_n - q^* \) with $\bar{q}_n$ being an local average over the previous 1,000 iterations. The error is then measured by the $\infty$-norm. 
One finds that the Anderson mixing does improve the convergence. But the improvement does not seem to be overwhelming.  
\begin{figure}
    \centering
    \includegraphics[scale=0.2]{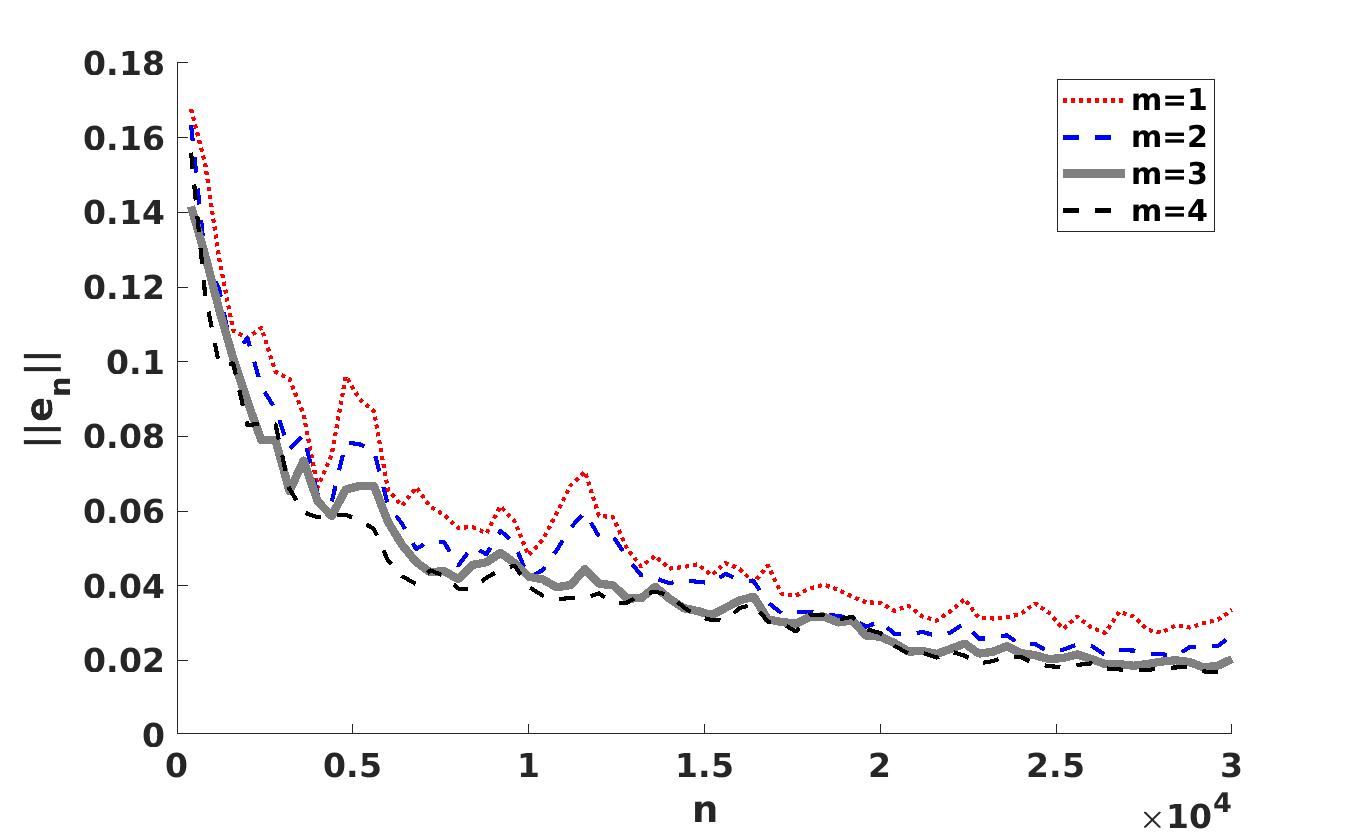}
    \caption{The error $\|\bar{q}_n-q^*\|_{\infty}$ from the Anderson mixing method with $m=2,3,4$ and $5$ using $a_n= [50 + 4 n^{3/4} ]^{-1}$.}
    \label{fig:anderson}
\end{figure}

Figure \ref{fig:bi} plots the mixing coefficients $b_i$'s from the stochastic Anderson method with $m=3.$ It can be observed that these coefficients are stochastic in nature. Remarkably, after a short burn-in period, these coefficients tend to fluctuate around the uniform mean $1/m.$ One interpretation is that as the iterates $q_n$ get closer to the $q^*$, the residual error $G(q)$ in \eqref{eq: residual} is dominated by the stochastic error $\xi_n$. In this case the least-square problem is mostly determined by noise, and it does not show bias toward a particular step.
\begin{figure}
    \centering
    \includegraphics[scale=0.3]{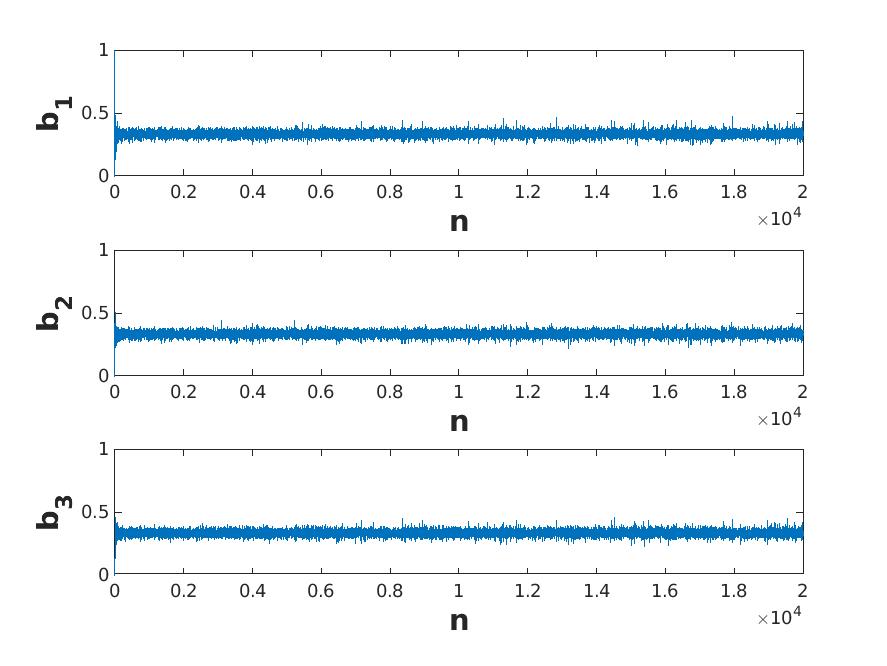}
    \caption{The coefficient $b_i$ from the Anderson (2) method }
    \label{fig:bi}
\end{figure}

 Lastly, all the computations were performed in Matlab R2020b, with parameters extracted from DFTB \cite{elstner1998self}. To follow up the discussion in Section \ref{cmp}, with 800 atoms, each stochastic iteration takes CPU time 4.75 (seconds), while a direct method takes 52.88 s. When the system size is increased to 1600 atoms, the respective CPU time is 10.12 s and 493.42 s.
  
\section{Conclusion}
This paper is motivated by the observation that
the main roadblock for extending electronic structure calculations to large systems is the SCF and the full diagonalizations that are involved in  each step of the procedure. This observation, for instance, has motivated a great deal of effort to develop linear or sublinear-scaling algorithms that do not directly rely on direct eigevalue computations \cites{goedecker1999linear,bowler2002recent,garcia2007sub}. Meanwhile,  stochastic algorithms have shown promising capability to handle linear and nonlinear problems in numerical linear algebra \cite{martinsson2020randomized}, and computational chemistry \cites{reynolds1982fixed,morales2021frontiers,golse2019random,li2020random,hermann2020deep}. This paper takes an initial step toward a stochastic implementation of the SCF. The main purpose is to establish certain convergence results. In particular, we showed that when the mixing parameters are selected a priori, the mixing method converges with probability one when the iteration is stable. Additionally, we derived the concentration inequality for the stochastic method. Some of these results are similar to those from the stochastic gradient descent methods in machine learning \cites{bottou2018optimization,johnson2013accelerating,defazio2014saga,nguyen2017sarah}. A crucial issue in the current approach is the stability: Since the contractive property only holds in the vicinity of the solution,  one must establish the stability of the iterations before proving the convergence.

While the convergence is a critical issue, many practical aspects remain as open issues, and they were not  studied in this paper. First, how to choose the mixing parameter $b_i$'s  in advance still remains open. Although we have shown the dependence of the error bound on $m$ and the mixing parameters, our analysis does not provide a clear criterion. Secondly, the choice of the damping parameter has a direct impact on the convergence. It would be of practical importance to be able to adjust them on-the-fly, as studied in the machine learning literature \cite{bottou2018optimization}. Finally,  as discussed in Section \ref{cmp}, comprehensive studies are needed to compare the stochastic algorithms to direct implementations of SCF to evaluate the performance for different physical systems. On the other hand, the two approaches do not have to be mutually exclusive in practice. For instance,  one can run iterations using stochastic algorithms first, and later switch to a direct method to improve the accuracy at the final stage. Such a strategy is used in machine learning, i.e., the stochastic variance reduction gradient \cite{johnson2013accelerating}. For example, Reddi et al \cite{reddi2016stochastic} showed that alternating between stochastic and direct methods can be better than using only direct methods.

\section{Acknowledgment}
The authors thank Prof. Kieron Burke and Prof. Lin Lin for discussions and references related to this work.

\appendix

\section{The diagonal estimator}

We restate the stochastic framework \cite{bekas2007estimator} and the property of the Hutchinson estimator \cite{avron_randomized_2011} in the following lemma.
\begin{lemma}\label{Diagonal}
For each matrix $A \in \mathbb{R}^{M\times M}$, the follow identity holds,
\begin{equation}
 \mathrm{diag}(A) = \mathrm{diag}(\mathbb{E}[Avv^T]),
\end{equation}
where $v\in \mathbb{R}^M$ is a random vector satisfying,
\begin{equation}
    \mathbb{E}[v v^T] =I_{M\times M}. 
\end{equation}Moreover, if the Hutchinson estimator is used, then
\begin{equation*}
    \mathrm{Var}(\mathrm{diag}[Avv^T])=\|A\|_F^2-\sum A_{ii}^2,
\end{equation*}where the entries of $v$ are i.i.d Rademacher random variables, 
\begin{equation*}
    \mathbb{P}\{v^{(i)}=\pm 1\}=\frac{1}{2}.
\end{equation*}
\end{lemma}
The {\it diagonal} can be estimated using a Monte-Carlo sum with $n_{vec}$ i.i.d. random vectors. A variety of such estimators are investigated in \cite{avron_randomized_2011}.

\section{An error estimate on the Krylov subspace approximation}\label{sec: A1}

Let $\ell'>\ell>1$ and $p_{\ell'}(x)$ be the Chebyshev polynomial approximation of degree $\ell'$ to $f(x)$. We recall that $\ell$ is the number of iterations for using the Lanczos algorithm.
By the triangle inequality, we split the error into three terms,
\begin{equation*}
    \begin{aligned}
        &\| f(A)v -  \|v\|_2 V_{\ell}f(T_{\ell})e_1 \|_2\leq \| f(A)v -  p_{\ell'}(A)v \|_2+\| p_{\ell'}(A)v -  \|v\|_2 V_{\ell}p_{\ell'}(T_{\ell})e_1 \|_2\\
        &+\| \|v\|_2 V_{\ell}p_{\ell'}(T_{\ell})e_1 -  \|v\|_2 V_{\ell}f(T_{\ell})e_1 \|_2.
    \end{aligned}
\end{equation*}
We will derive upper bounds for these three terms. For the first and third terms, we use Theorem $7.2$ in \cite{trefethen2019approximation}. Meanwhile, we will Theorem $8.1$ in \cite{trefethen2019approximation} to find an upper bound for the second term.  
\begin{theorem}[Theorem 7.2 in \cite{trefethen2019approximation}]
For an integer $s\geq 1$, let $f$ and its derivatives through $f^{(s-1)}$ be absolutely continuous on $[-1,1]$ and suppose that the $s$th order derivative $f^{(s)}$ is of bounded variation $V$. Then, for any $\ell>s$, the Chebyshev approximation of degree $\ell$, $p_{\ell}$, satisfies, 
\begin{equation*}\label{chebyapprox}
    \|f-p_{\ell}\|\leq \frac{2V}{\pi s(\ell-s)^s},
\end{equation*}where $\|h\|$ denotes the supremum norm of the function $h$.
\end{theorem}
For a general symmetric matrix $A$ whose spectrum is not necessarily contained in $[-1,1]$, a linear transformation is first applied to $A$ to shift the spectrum to the interval $[-1,1]$. This can be achieved using the linear transformation
\begin{equation*}
    A\mapsto \Tilde{A}:=\frac{2A}{b-a}-\frac{b+a}{b-a} I, \quad b:= \lambda_{\max}, \; a:=\lambda_{\min}.
\end{equation*} With this transformation, we have,
\begin{equation}\label{lineartrans}
    f(x)\mapsto \tilde{f}(x):=f(\frac{b-a}{2}x+\frac{a+b}{2})\approx \tilde{p}(x)\mapsto p(x):=\tilde{p}(\frac{2x}{b-a}-\frac{a+b}{b-a}),
\end{equation}which means that $p(x)$ is the Chebyshev approximation of $f(x)$ defined on the desired interval. Following this the procedure, the variation of $\tilde{f}^{(s)}(x)$ is proportional to that of $f^{(s)}(x)$ as follows
\begin{equation*}
    \tilde{V}=\bigg(\frac{b-a}{2}\bigg)^sV,
\end{equation*}where $V$ is the variation of $f^{(s)}(x)$.

Since $A$ is symmetric as defined in \eqref{A},  a direct application of the above theorem yields,
\begin{equation*}
    \|f(A)-p_{\ell'}(A)\|=\max_{\lambda\in\sigma(A)}|f(\lambda)-p_{\ell'}(\lambda)|\leq\|f-p_{\ell'}\|=\|\tilde{f}-\tilde{p}_{\ell'}\|\leq \frac{V}{2^{s-1}\pi s}\bigg(\frac{b-a}{\ell'-s}\bigg)^s.
\end{equation*}Consequently, we have 
\begin{equation*}
    \|f(A)v-p_{\ell'}(A)v\|_2\leq \frac{\|v\|_2V}{2^{s-1}\pi s}\bigg(\frac{b-a}{\ell'-s}\bigg)^s.
\end{equation*}
Similarly, we bound the third term as follows
\begin{equation*}
    \| \|v\|_2 V_{\ell}p_{\ell'}(T_{\ell})e_1 -  \|v\|_2 V_{\ell}f(T_{\ell})e_1 \|_2\leq \frac{\|v\|_2V}{2^{s-1}\pi s}\bigg(\frac{b-a}{\ell'-s}\bigg)^s,
\end{equation*}since $V_{\ell}$ is the semi-orthogonal matrix whose $2$-norm is $1$.
To estimate the second term, we use Theorem $8.1$ in \cite{trefethen2019approximation}, which relies on the Bernstein ellipse.

\begin{theorem}\label{thmb4}[Theorem 8.1 \cite{trefethen2019approximation}]
Let $f(x)$ be analytic in $[-1,1]$ and assume that $f(x)$ can be extended analytically to the open Bernstein ellipse $E_{\rho}$ for some $\rho>1$, where it satisfies $|f(x)|\leq M(\rho)$ for some $M(\rho)$. Then, the coefficients of the Chebyshev approximation of the function satisfy $|c_0|\leq M(\rho)$ and 
\begin{equation*}
    |c_n|\leq 2M(\rho)\rho^{-n},\quad n\geq 1.
\end{equation*}
\end{theorem}Note that the Fermi-Dirac distribution $f(x)$ is analytic in the strip $\{z:|\text{Im}(z)|<\frac{\pi}{\beta}\}$ and $z=\mu\pm\frac{\pi}{\beta}i$ are the singular points. Since the two singular points correspond to $\frac{2}{b-a}\big(\mu-\frac{a+b}{2}\pm\frac{\pi}{\beta}i\big)$ under the linear transformation, the function $\tilde{f}(x)$ is analytic in the scaled strip $\{z:|\text{Im}(z)|<\frac{2}{b-a}\frac{\pi}{\beta}\}$. Thus, by the continuity of the Bernstein ellipse $E_{\rho}$ with respect to $\rho$, we can find $\rho$ sufficiently close to $1$ such that a Bernstein ellipse is a proper subset of the strip. Then, we apply the theorem to the function $\tilde{f}(x)$ and consider its Chebyshev approximation $\tilde{p}_{\ell'}(x)=\sum_{n=0}^{\ell'}c_nT_n(x)$. By the scaling in \eqref{lineartrans}, we can deduce that
\begin{equation*}
    p_{\ell'}(A)=\tilde{p}_{\ell'}(\tilde{A}).
\end{equation*}
Thus, we have
\begin{equation*}
\begin{aligned}
    &\| p_{\ell'}(A)v -  \|v\|_2 V_{\ell}p_{\ell'}(T_{\ell})e_1 \|_2=\|\sum_{n=\ell+1}^{\ell'}c_nT_n(\Tilde{A})v+\|v\|_2V_{\ell}\sum_{n=\ell+1}^{\ell'}c_nT_n(\tilde{T}_{\ell})e_1\|_2\\
    &\leq 2\sum_{n=\ell+1}^{\ell'}|c_n|\|v\|_2\leq 2\sum_{n=\ell+1}^{\infty}|c_n|\|v\|_2=4M(\rho)\|v\|_2\frac{\rho^{-\ell}}{\rho-1}.  
\end{aligned}
\end{equation*}In the first equality, we applied  Lemma $3.1$ in \cite{saad1992analysis}, which states as 
\begin{equation*}
    p_j(A)v=\|v\|_2V_\ell p_j(T_\ell)e_1
\end{equation*}for any polynomial $p(x)$ of degree $j\leq \ell$. In addition, the first inequality holds since $|T_n(x)|\leq 1$ and $\|\tilde{T_{\ell}}\|\leq\|\tilde{A}\|\leq 1$. In the last step, we have used the theorem above. 

Now we are ready to prove Theorem \ref{krylovapprox}.
\begin{proof} By collecting the above results, the error is bounded by,
\begin{equation*}
    \begin{aligned}
        &\| f(A)v -  \|v\|_2 V_{\ell}f(T_{\ell})e_1 \|_2\leq \frac{\|v\|_2V}{2^{s-2}\pi s}\bigg(\frac{\lambda_{\max}(A)-\lambda_{\min}(A)}{\ell'-s}\bigg)^s+4M(\rho)\|v\|_2\frac{\rho^{-\ell}}{\rho-1}.
    \end{aligned}
\end{equation*}Recalling that $\ell'$ is arbitrary and greater than $\ell$, the first term on the right hand side can be removed by letting $\ell'\rightarrow \infty$. This completes the proof. 
\end{proof}

\section{The proof of Theorem \ref{thmcontract'}}\label{sec: A2}

\begin{proof}

By the assumption that $\textrm{Re}(\lambda)<1$ for each eigenvalue $\lambda$ of $K'(q^*)$, there exists a small $\theta_{\max}>0$ such that for any $0<\theta<\theta_{\max}$, the spectral radius of $K_{\theta}'(q^*)$ is less than $1$. This is similar to the standard stability condition for the Euler's method for solving ODEs, and it corresponds to a circular disk in  the complex plane. We choose such a parameter $\theta$. We observe that $K_{U,\theta}'(p^*)$ in the theorem \ref{thmcontract'} reduces to an upper triangular matrix from the Schur decomposition of $K'(q^*)$ with some unitary matrix $U$. 
Theorems 3 and 4 in \cite{kincaid2009numerical}[p 214] provide an explicit construction of a matrix norm of $K_{U,\theta}'(p^*)$ that is less than 1.  The norm is induced by an inner product using a diagonal matrix, $D\succ 0.$ We denote the vector norm by  $\|\cdot\|_D.$  Those theorems and the choice of $\theta$ yield that, $\forall p_1 $ and $p_2,$
\begin{equation}
    \|K_{U,\theta}'(p^*) (p_1 - p_2) \|_D\leq c_1\|p_1 - p_2 \|_D,
\end{equation}for some $c_1<1.$ Now, we are ready to prove the result. Since $K$ is continuously differentiable, $K_{U,\theta}$ in Theorem \ref{thmcontract'} is continuously differentiable as well. Thus, by the continuity of the Jacobian, we have  $\|K_{U,\theta}'(p)\|_D<c_2<1,$ for some $c_2$ and for any $p\in B_D(p^*,\rho)$ where $B_D(p^*,\rho)$ is a ball centered at $p^*$ of some radius $\rho$ with respect to the $D$-norm. The local contractiveness can then be checked using the mean-value theorem for vector-valued functions. Since $\theta$ is arbitrarily chosen in $(0,\theta_{\max})$, the proof is completed.

\end{proof}

\section{Lemmas for the proofs in section \ref{sectionsto}}
\label{sec: A5}

\begin{lemma}\label{positive}
For each $N\in\mathbb{N}$ and a positive number $v$, assume that $\mathbb{P}(\|e_n\|_2>v\textrm{ for all }n\geq N)=0$. Then, $$\mathbb{P}(\liminf_{n}\|e_n\|_2\leq v)=1.$$ 
\end{lemma}
\begin{proof}
Let $A_N:=\{w:\|e_n\|_2>v\textrm{ for all }n\geq N\}$. Note that $A_N\subset A_{N+1}$. Then, \begin{equation*}
    \bigcup_NA_N=\{w:\textrm{there exists a }N\in\mathbb{N}\textrm{ such that }\|e_n\|_2>v\textrm{ for all }n\geq N\},
\end{equation*}which implies
\begin{equation*}
    \bigg(\bigcup_NA_N\bigg)^c=\{w:\liminf_n\|e_n\|_2\leq v\}.
\end{equation*}By the countable additivity,  therefore, the Lemma holds true.
\end{proof}

However, for case $m\geq 2$, we will develop a more sophisticated tool. In the following Lemma, we employ well known results on irreducible aperiodic stochastic matrices in \cites{norris1998markov,lefebvre2007applied}. Moreover, we will use the Perron-Frobenius theorem in \cite{meyer2000matrix}.

\begin{lemma}\label{genconv}
Suppose that a sequence of iterates $\{e_n\}$ from the linear mixing scheme \ref{alg: linear} is bounded. Assume that for fixed $b_m>0$,
\begin{equation}\label{converges}
    \lim_{n\to\infty}\left[\|e_{n+m-1}\|_2^2+\sum_{j=2}^m\bigg(\sum_{i=1}^{m-j+1}b_i\bigg)\|e_{n+m-j}\|_2^2\right]=x.
\end{equation}Then,
\begin{equation*}
    \lim_{n\to\infty}\|e_n\|_2=\sqrt{\frac{x}{1+\sum_{j=2}^m\bigg(\sum_{i=1}^{m-j+1}b_i\bigg)}}.
\end{equation*}
\end{lemma}
\begin{proof}
Take a subsequence $\{e_{n_k}\}$. Then, since the sequence $\{e_n\}$ is bounded, we can find a convergent sub-subsequence. To reduce notations, we preserve the same indices $\{n_k\}$ for this convergent sub-subsequence. Furthermore, we can assume that the $m$ shifted sequences are convergent, namely, 
\begin{equation*}
    \{e_{n_k}\},\{e_{n_k-1}\},...,\{e_{n_k-(m-1)}\}\textrm{ converge}.
\end{equation*}

For any $N\in\mathbb{N}$ and $N>m$, let $l_{i,N}:=\lim_{k\to\infty}e_{n_k-N+i}$ for $1\leq i \leq N$. Then, it follows that for $n\in\{1,2,...,N-m\}$,
\begin{equation}\label{mixingscheme}
    l_{n+m,N}=\sum_{i=1}^mb_il_{n+i-1,N},
\end{equation}from the mixing scheme \ref{alg: linear} by noting that the damping parameters $\{a_n\}$ converge to $0$. We claim that the limits of the shifted sequences are the same, i.e.,
\begin{equation*}
    \lim_{k\to\infty}e_{n_k}=\lim_{k\to\infty}e_{n_k-1}=\cdots=\lim_{k\to\infty}e_{n_k-(m-1)}.
\end{equation*}
It is sufficient to show that the first entries of the limits are the same. With the standard basis vector $e_1=(1,0,0,..,0)^T$, we define the $n$th vector 
\begin{equation*}
    \tilde{l}_{n,N}:=(l_{n+m-1,N}\cdot e_1,l_{n+m-2,N}\cdot e_1,...,l_{n,N}\cdot e_1)^T,
\end{equation*}which contains the first entries of the vectors $\{l_{i,N}\}_{i=n}^{n+m-1}$.
Next, from the relation \eqref{mixingscheme}, we can define a recursive system such that for $1\leq n\leq N-m$,
\begin{equation*}
    \tilde{l}_{n+1,N}=B\tilde{l}_{n,N},
\end{equation*}where
\begin{equation*}
    B=\begin{pmatrix}b_m&b_{m-1}&\cdots&b_2&b_1\\
    1&0&\cdots&0&0\\0&1&\cdots&0&0\\\vdots&\vdots&\cdots&\cdots&0\\0&\cdots&\cdots&1&0\end{pmatrix}.
\end{equation*}The matrix $B$ can be recognized as a companion matrix. Thus, the characteristic polynomial of $B$ has one as its root, because $\sum_{i=1}^mb_i=1$. 

Note that it is the mixing scheme with $m$ steps, which assumes that $b_1>0$. For this reason, $B$ is an irreducible matrix, which means that all the nodes $\{1,2,..,m\}$ communicate in the graph corresponding to the matrix $B$ \cite{lefebvre2007applied}[p 86], which can interpreted as a transition matrix.

Moreover, since $b_m>0$ by assumption and $B$ is irreducible, $B$ is aperiodic \cite{lefebvre2007applied}[p 91]. Also, since all rows sum to one, $B$ is a stochastic matrix. By the Gershgorin's theorem, we can guarantee that $\rho(B)\leq 1$, which denotes the spectral radius of $B$. Thus, by applying the Perron-Frobenius theorem to $B^T$ \cite{meyer2000matrix}[p 673], we can find a left eigenvector $\pi>0$ 
\begin{equation*}
    \pi B=\pi.
\end{equation*}Since $B$ is irreducible, aperiodic and has the invariant distribution $\pi$, the matrix $B^n$ converges to equilibrium as stated in \cite{norris1998markov}[Theorem 1.8.3], namely,
\begin{equation*}
    \lim_{n\to\infty}B^n=  \begin{pmatrix}\pi_1&\pi_2&\cdots&\pi_m\\
    \pi_1&\pi_2&\cdots&\pi_m\\\pi_1&\pi_2&\cdots&\pi_m\\\vdots&\vdots&\cdots&\vdots\\\pi_1&\pi_2&\cdots&\pi_m\end{pmatrix}.
\end{equation*}On the other hand, from the recursive relation, we have $\tilde{l}_{N-m+1,N}=B^{N-m}\tilde{l}_{1,N}$, or 
\begin{equation*}
    \begin{pmatrix}
    l_{N,N}\cdot e_1\\l_{N-1,N}\cdot e_1\\ \vdots\\l_{N-m+1,N}\cdot e_1
    \end{pmatrix}=B^{N-m}\begin{pmatrix}
    l_{m,N}\cdot e_1\\l_{m-1,N}\cdot e_1\\ \vdots\\l_{1,N}\cdot e_1
    \end{pmatrix}.
\end{equation*}Since the sequence $\{e_n\}_{n=1}^{\infty}$ is bounded by assumption and $B^n$ converges, by letting $N\to\infty$, we can obtain the result that for any $i,j\in\{0,1,2,...,m-1\}$,
\begin{equation*}
        \lim_{N\to\infty}l_{N-i,N}\cdot e_1=\lim_{N\to\infty}l_{N-j,N}\cdot e_1,
\end{equation*}which implies that
\begin{equation*}
    \lim_{k\to\infty}e_{n_k}\cdot e_1=\lim_{k\to\infty}e_{n_k-1}\cdot e_1=\cdots=\lim_{k\to\infty}e_{n_k-m+1}\cdot e_1.
\end{equation*}With the same technique for the other coordinates, it immediately follows that  
\begin{equation*}
    \lim_{k\to\infty}e_{n_k}=\lim_{k\to\infty}e_{n_k-1}=\cdots=\lim_{k\to\infty}e_{n_k-m+1}.
\end{equation*}
Here, we proved the claim. From this result, we can use Assumption \eqref{converges} as follows
\begin{equation*}
    x=\lim_{k\to\infty}\left[\|e_{n_k}\|_2^2+\sum_{j=2}^m\bigg(\sum_{i=1}^{m-j+1}b_i\bigg)\|e_{n_k+1-j}\|_2^2\right]=\bigg[1+\sum_{j=2}^m\bigg(\sum_{i=1}^{m-j+1}b_i\bigg)\bigg]\|\lim_{k\to\infty}e_{n_k}\|_2^2.
\end{equation*}To sum up, for any subsequence of the sequence $\{\|e_n\|_2\}_{n=1}^{\infty}$, we can find a sub-subsequence convergent to \begin{equation*}
    \sqrt{\frac{x}{1+\sum_{j=2}^m\bigg(\sum_{i=1}^{m-j+1}b_i\bigg)}}.
\end{equation*}Therefore, this completes the proof of the lemma.

\end{proof}

Next, in order to derive non-negative supermartingales from the general mixing scheme \ref{alg: linear} with the Lyapunov functions \eqref{mLya}, we prove the following inequalities which will be used in Theorem  \ref{casem}.

By using the Cauchy-Schwarz inequality and the Jensen's inequality, the $n+1$st error can be bounded as
\begin{equation}\label{uniform}
\begin{split}
    &\mathbb{E}[\|e_{n+1}\|_2^2]=\mathbb{E}[\|B_m(e_n)+a_nB_m(G_n)\|_2^2]\\
    &\leq \sum_{i=1}^mb_i^2\mathbb{E}[\|e_i+a_nG_i\|_2^2]+\sum_{i\neq j}\mathbb{E}[b_ib_j(e_i+a_nG_i,e_j+a_nG_j)]\\
    &\leq \sum_{i=1}^mb_i^2\mathbb{E}[\|e_i+a_nG_i\|_2^2]+\sum_{i\neq j}\sqrt{b_i^2\mathbb{E}[\|e_i+a_nG_i\|_2^2]}\sqrt{b_j^2\mathbb{E}[\|e_j+a_nG_j\|_2^2]}\\
    &=\bigg(\sum_{i=1}^mb_i\sqrt{\mathbb{E}[\|e_i+a_nG_i\|_2^2]}\bigg)^2\leq \sum_{i=1}^mb_i\mathbb{E}[\|e_i+a_nG_i\|_2^2]
\end{split}
\end{equation}With this, we can establish inequalities which will define a non-negative supermartingale eventually.

For general $m$, from the inequality \eqref{uniform}, the conditional expectation of $\|X_{n+1}\|_{n+1}$ is bounded as
\begin{equation*}
    \begin{split}
        &\mathbb{E}[\|X_{n+1}\|_{n+1}|X_n]\\
        &=\mathbb{E}[\|e_{n+m}\|_2^2|X_n]+\sum_{j=2}^m\sum_{i=j}^mb_{m-i+1}\mathbb{E}[\|e_{n+j-1+m-i}+a_{n+m-2+j}G_{n+j-1+m-i}\|_2^2|X_n]\\
        &\leq \sum_{i=1}^mb_{i}\mathbb{E}[\|e_{n+i-1}+a_{n+m-1}G_{n+i-1}\|_2^2|X_n]\\
        &+\sum_{j=2}^m\sum_{i=j}^mb_{m-i+1}\mathbb{E}[\|e_{n+j-1+m-i}+a_{n+m-2+j}G_{n+j-1+m-i}\|_2^2|X_n].
    \end{split}
\end{equation*} In this upper bound, the first summation is divided into two parts as follows
\begin{equation*}
        b_m\mathbb{E}[\|e_{n+m-1}+a_{n+m-1}G_{n+m-1}\|_2^2|X_n]+\sum_{i=1}^{m-1}b_i\|e_{n+i-1}+a_{n+m-1}G_{n+i-1}\|_2^2.
\end{equation*}
Meanwhile, the double summation term equals
\begin{equation*}
    \begin{split}
        &\sum_{j=2}^mb_{m-j+1}\mathbb{E}[\|e_{n+m-1}+a_{n+m-2+j}G_{n+m-1}\|_2^2|X_n]\\
        &+\sum_{j=2}^{m-1}\sum_{i=j+1}^mb_{m-i+1}\|e_{n+j-1+m-i}+a_{n+m-2+j}G_{n+j-1+m-i}\|_2^2,
    \end{split}
\end{equation*}by pulling out the terms involving $e_{n+m-1}$. By grouping terms with and without $e_{n+m-1}$, we can rewrite the bound as
\begin{equation*}
    \begin{split}
        &\mathbb{E}[\|X_{n+1}\|_{n+1}|X_n]\leq\sum_{j=1}^mb_{m-j+1}\mathbb{E}[\|e_{n+m-1}+a_{n+m-2+j}G_{n+m-1}\|_2^2|X_n]\\
        &+\sum_{j=2}^{m}\sum_{i=j}^mb_{m-i+1}\|e_{n+j-2+m-i}+a_{n+m-3+j}G_{n+j-2+m-i}\|_2^2.
    \end{split}
\end{equation*}

 By recalling Definition \eqref{lyapunovfunction}, it follows that
\begin{equation}\label{supermartingm}
\begin{split}
    &\mathbb{E}[\|X_{n+1}\|_{n+1}|X_n]-\|X_n\|_n\\
    &=\sum_{j=1}^mb_{m-j+1}\mathbb{E}[\|e_{n+m-1}+a_{n+m-2+j}G_{n+m-1}\|_2^2|X_n]-\|e_{n+m-1}\|_2^2\\
    &\leq 2\bigg(\sum_{j=1}^mb_{m-j+1}a_{n+m-2+j}\bigg)(e_{n+m-1},R_{n+m-1})\\
    &+\sum_{j=1}^mb_{m-j+1}a_{n+m-2+j}^2\bigg(\Xi+\|R_{n+m-1}\|^2\bigg)\\
    &\leq -(1-c)\bigg(\sum_{j=1}^mb_{m-j+1}a_{n+m-2+j}\bigg)\|e_{n+m-1}\|_2^2\\
    &+\Xi\sum_{j=1}^mb_{m-j+1}a_{n+m-2+j}^2.
\end{split}
\end{equation}In the first inequality, we used Assumption \ref{as:05} and the fact that $(R_{n+m-1},\xi_{n+m-1})=0$ together with Definition \eqref{eq: residual}. In the second inequality, we apply the inequalities \eqref{removeresidue} and the condition of boundedness in \eqref{eq: conda}. From the last part of the inequality \ref{supermartingm}, we make notation for the coefficients attached to the terms $\|e_{n+m-1}\|^2$ and $\Xi$ as follow
\begin{subequations}\label{newcoeff}
\begin{eqnarray}
        &A_n:=\sum_{j=1}^mb_{m-j+1}a_{n+m-2+j}\\
        &\chi_n:=\sum_{j=1}^mb_{m-j+1}a_{n+m-2+j}^2.
\end{eqnarray}
\end{subequations}Note that the condition \eqref{eq: conda} together with \eqref{mixingcoeff} yields
\begin{equation}
    \sum_{n=1}^{\infty}A_n=\infty,\quad\sum_{n=1}^{\infty}\chi_n<\infty.
\end{equation}

By defining $V_n(X_n)=\|X_n\|_n+\Xi \sum_{i=n}^{\infty}B_i$, the above can be rewritten 
\begin{equation}\label{lyapunovfunction}
    \mathbb{E}[V_{n+1}(X_{n+1})|X_n]-V_n(X_n)\leq -(1-c)A_n\|e_{n+m-1}\|_2^2\leq 0.
\end{equation}

\section{The proofs of theorems in section \ref{section4}}\label{sec: A6} 

\subsection{Theorem \ref{stabsimple}}
\begin{proof} 

To establish the stability, we follow the proof in \cite{kushner2003stochastic}[p 112, Theorem 5.1]. 
By the definition of $V(\cdot)$, for any $q_n\in B(q^*,\rho)$ direct calculations yield 
\begin{equation*}
\begin{aligned}
    &\mathbb{E}_n\left[V(q_{n+1})\right]-V(q_n)=\mathbb{E}_n\left[\|e_n+a_nG_n\|_2^2\right]-\|e_n\|_2^2,\\
    &=2a_n\mathbb{E}_n\left[(e_n,G_n)\right]+a_n^2\mathbb{E}_n\left[\|G_n\|_2^2\right],\\
    &\leq -2a_n(1-c)\|e_n\|_2^2+(c+1)^2a_n^2\|e_n\|_2^2+\Xi a_n^2,\\
    &\leq -a_n(1-c)\|e_n\|_2^2+\Xi a_n^2.
\end{aligned}
\end{equation*}In the first inequality, we use Assumption \ref{as:05} and remove the cross term $\mathbb{E}_n[(R_n,\xi_n)]$, where the residual $R_n$ and the error $\xi_n$ sum to $G_n$. Besides, we used the following inequalities
\begin{subequations}\label{removeresidue}
\begin{eqnarray}
        &(e_n,R_n)=-\|e_n\|_2^2+(K(q_n)-K(q^*),e_n)\leq (c-1)\|e_n\|_2^2\\
        &\|R_n\|^2\leq (1+c)^2\|e_n\|_2^2.
\end{eqnarray}
\end{subequations}
In the last step, we used the condition that $a_n\leq\frac{1-c}{(1+c)^2}$. 

We proceed by observing that $V_n(q_n)\geq 0$ and
\begin{equation*}
    \delta V_{n+1}-\delta V_n=-\Xi a_n^2,
\end{equation*}
which implies the following inequality,
\begin{equation}\label{supmartingineq}
    \mathbb{E}_n[V_{n+1}(q_{n+1})]-V_n(q_n)\leq -(1-c)a_n\|e_n\|_2^2\leq 0.
\end{equation}

Here, we define the stopping time $\tau_{\rho}:=\{n:\|e_n\|_2>\rho\}$. Accordingly, we define a stopped process for $q_n$ and a corresponding Lyapunov function as follows 
\begin{equation}\label{stoppedpro}
    \tilde{q}_n:=\begin{cases}
    q_n,\; n\leq\tau_{\rho}\\
    q_{\tau_{\rho}},\; n>\tau_{\rho}
    \end{cases},
    \tilde{V}_n:=\begin{cases}
    V_n,\; n\leq\tau_{\rho}\\
    V_{\tau_{\rho}},\; n>\tau_{\rho}
    \end{cases}.
\end{equation}This technique is shown in the proof of Theorem 5.1 in \cite{kushner2003stochastic}, which yields the non-negative supermartingale $\{\tilde{V}_n(\tilde{q}_n)\}$. 

Similar to the proof of Theorem 5.1 in \cite{kushner2003stochastic}, we can deduce that,
\begin{equation*}
\begin{aligned}
   &\mathbb{P}\left\{\sup_{n}\|e_n\|_2 > \rho|q_1\right\}\mathbb{I}_{\{q_1\in B(q^*,\rho)\}}\leq\mathbb{P}\left\{\sup_n V_n(q_n)>\rho^2|q_1\right\}\mathbb{I}_{\{q_1\in B(q^*,\rho)\}}\leq \frac{V_1(q_1)}{\rho^2},
\end{aligned} 
\end{equation*}
which concludes the first part of the theorem. 
\medskip

Since the stopped process $\{\tilde{V}_n(\tilde{q}_n)\}_{n\ge 1}$ forms a supermartingale as
\begin{equation}\label{super}
    \mathbb{E}_n[\tilde{V}_{n+1}(\tilde{q}_{n+1})]\leq \tilde{V}_n(\tilde{q}_n),~ \forall n \geq 1,
\end{equation}
$\Tilde{V}_n(\tilde{q}_n)$ converges to some random variable $\tilde{V}\geq 0$. In the event where $\|e_n\|_2\leq\rho$ for all $n\in\mathbb{N}$, this implies that
\begin{equation*}
    \lim_{n\to\infty}V_n(q_n)=\lim_{n\to\infty}\|e_n\|_2^2,
\end{equation*}with probability one, since $\sum_n a_n^2<\infty$. Suppose that $\|e_n\|_2$ converges to a positive random variable $V$ with positive probability. Then, there exists a positive number $\delta>0$ such that $$\mathbb{P}\left\{\lim_{n\to\infty}\|e_n\|_2>\delta\Big|\{q_n\}\subset B(q^*,\rho)\right\}>0.$$ By Lemma \ref{positive}, we have for some $N\in\mathbb{N}$,
\begin{equation*}
    \mathbb{P}\left\{\|e_n\|_2>\frac{\delta}{2}\textrm{ for all }n\geq N \Big|\lim_{n\to\infty}\|e_n\|>\delta,\{q_n\}\subset B(q^*,\rho)\right\}>0.
\end{equation*}

On the other hand, by a telescoping trick with the inequality \eqref{supmartingineq}, for any given $q_1\in B(q^*,\rho)$, we have
\begin{equation*}
    V_1(q_1)\geq V_1(q_1)-\mathbb{E}_1[\Tilde{V}_n(\Tilde{q}_n)]\geq 2(1-c)\mathbb{E}_1\left[\sum_{i=1}^{n-1}a_{i}\|\tilde{e}_{i}\|_2^2\right],
\end{equation*}which implies
\begin{equation*}
    \mathbb{E}_1\left[\sum_{i=1}^{\infty}a_{i}\|\tilde{e}_{i}\|_2^2\right]<\infty.
\end{equation*}

By the above results, we can deduce that
\begin{equation*}
    \mathbb{P}\left\{\|e_n\|_2\geq\frac{\delta}{2}\textrm{ for all }n\geq N,\{q_n\}_{n=1}^{\infty}\subset B(q^*,\rho)\right\}>0,
\end{equation*}which implies that
\begin{equation*}
\begin{split}
    &\infty>\mathbb{E}_1\left[\sum_{i=1}^{\infty}a_{i}\|\tilde{e}_{i}\|_2^2\right]\geq\mathbb{E}_1\left[\sum_{i=1}^{\infty}a_{i}\|\tilde{e}_{i}\|_2^2\mathbb{I}_{\{\|e_n\|_2>\frac{\delta}{2}\textrm{ for all }n\geq N,\{q_n\}_{n=1}^{\infty}\subset B(q^*,\rho)\}}\right]\\
    &\geq \frac{\delta^2}{4}\bigg(\sum_{i=N}^{\infty}a_{i}\bigg)\cdot\mathbb{P}\left\{\|e_n\|_2>\frac{\delta}{2}\textrm{ for all }n\geq N,\{q_n\}_{n=1}^{\infty}\subset B(q^*,\rho)\right\}.
\end{split}
\end{equation*}Since $\sum_na_n=\infty$, this is a contradiction. Therefore, $\|e_n\|_2$ converges to $0$ with probability one when $\{q_n\}\subset B(q^*,\rho)$.
\end{proof}

\subsection{Theorem \ref{casem}}

\begin{proof}
Define $\tau_{\rho}=\min\{n:\|e_{n+m-1}\|_2>\rho\}$ as a stopping time. Similar to \eqref{stoppedpro}, we define 
\begin{equation}\label{stoppedpro2}
    \tilde{X}_n:=\begin{cases}
    X_n,\; n\leq \tau_{\rho}\\
    X_{\tau_{\rho}},\; n>\tau_{\rho}
    \end{cases},
    \tilde{V}_n:=\begin{cases}
    V_n,\; n\leq \tau_{\rho}\\
    V_{\tau_{\rho}},\; n>\tau_{\rho}.
    \end{cases}
\end{equation}
We will prove the theorem similar to the proof of the theorem \ref{stabsimple}. The inequality \eqref{lyapunovfunction} will yield a non-negative supermartingale, which justifies the first statement. Also, the stopped process $\Tilde{V}_n(\tilde{X}_n)$ converges to some random variable $\tilde{V}\geq 0$. Therefore, in the event where $\|e_n\|_2\leq\rho$ for all $n\in\mathbb{N}$, we can deduce that  
\begin{equation*}
\begin{split}
    &\lim_{n\to\infty}V_n(X_n)=\lim_{n\to\infty}\|X_n\|_n=\lim_{n\to\infty}\left[\|e_{n+m-1}\|_2^2+\sum_{j=2}^m\sum_{i=j}^mb_{m-i+1}\|e_{n+j-2+m-i}\|_2^2\right]\\
    &=\lim_{n\to\infty}\left[\|e_{n+m-1}\|_2^2+\sum_{j=2}^m\bigg(\sum_{i=1}^{m-j+1}b_i\bigg)\|e_{n+m-j}\|_2^2\right]
\end{split}
\end{equation*}with probability one. The second equality holds as in the previous proof. Moreover, by applying the lemma \ref{genconv} to the last step, the sequence $\{\|e_n\|_2\}$ converges with probability one, namely, 
\begin{equation*}
    \lim_{n\to\infty}V_n(X_n)=\left[1+\sum_{j=2}^m\left(\sum_{i=1}^{m-j+1}b_i\right)\right]\lim_{n\to\infty}\|e_n\|_2^2
\end{equation*}

For similar reasoning in the proof of the theorem \ref{stabsimple}, $\|e_n\|_2$ converges to $0$ when all the iterates are in $B(q^*,\rho)$.

\end{proof}

\subsection{Theorem \ref{concentlinear}}\label{proofconcent} 

\begin{proof}
As already established, we use the almost supermartingale property \eqref{lyapunovfunction}. This property can be rewritten as
\begin{equation}
    \mathbb{E}[V_{n+1}(X_{n+1})|X_n]\leq V_n(X_n)-A_nk(X_n),
\end{equation}where $A_n$ is defined in \eqref{newcoeff} and $k(X_n)=(1-c)\|e_{n+m-1}\|_2^2$. With \eqref{stoppedpro2} and this function $k(X_n)$, we can define the non-negative supermartingale as similar in Theorem 5.1 \cite{kushner2003stochastic}
\begin{equation}
    \mathbb{E}[\tilde{V}_{n+1}(\tilde{X}_{n+1})|\tilde{X}_n]\leq \tilde{V}_n(\tilde{X}_n)-A_n\tilde{k}(\tilde{X}_n),
\end{equation}where 
\begin{equation}
\tilde{k}(\tilde{X}_n)=\begin{cases}
k(X_n),\; q_{n+m-1}\in B(q^*,\rho)\\
0,\; q_{n+m-1}\not\in B(q^*,\rho).
\end{cases}
\end{equation}

By taking the total expectation on the supermartingale and telescoping inequalities, we obtain
\begin{equation}\label{firstresult}
    \sum_{n=1}^jA_n\mathbb{E}[k(X_n)\mathbb{I}_{E_j}|X_1]\leq\sum_{n=1}^jA_n\mathbb{E}[\tilde{k}(\tilde{X}_n)|X_1]\leq V_1(X_1).
\end{equation}We note that as defined above, the function $k(X_n)$ is a convex function with respect to $q_{n+m-1}$ as a quadratic function. Thus, by the Jensen's inequality, we have
\begin{equation}
    (1-c)\|\bar{q}-q^*\|_2^2\leq \sum_{n=1}^j\left(\frac{A_n}{\sum_{n=1}^jA_n}\right)k(X_n),
\end{equation}where 
\begin{equation}
    \bar{q}:=\sum_{n=1}^j\left(\frac{A_n}{\sum_{n=1}^jA_n}\right)q_n.
\end{equation}To put these together, we arrive at
\begin{equation}
    \mathbb{E}[\|\bar{q}-q^*\|_2^2\mathbb{I}_{E_j}|X_1]\leq \frac{V_1(X_1)}{(1-c)\left(\sum_{n=1}^jA_n\right)}.
\end{equation}This is the first result. For convenience, let us denote the RHS of this inequality by $u_j$.

By applying the Markov's inequality to this result, we have
\begin{equation*}
    \mathbb{P}\left\{\|\bar{q}-q^*\|_2>\epsilon\textrm{ and }E_j\textrm{ occurs }|X_1\right\}\leq \frac{u_j}{\epsilon^2},
\end{equation*}equivalently,
\begin{equation*}
    \mathbb{P}\left\{\|\bar{q}-q^*\|_2\leq\epsilon\textrm{ or }E_j\textrm{ does not occur}|X_1\right\}\geq 1-\frac{u_j}{\epsilon^2}.
\end{equation*}

By the stability result in the theorem \ref{casem}, the probability for $E_j$ to not occur is bounded above
\begin{equation*}
    \mathbb{P}\left\{E_j\textrm{ does not occur}|X_1\right\}\leq \frac{V_1(X_1)}{\rho^2},
\end{equation*}which leads to 
\begin{equation*}
    \mathbb{P}\left\{\|\bar{q}-q^*\|_2\leq\epsilon|X_1\right\}\geq 1-\frac{u_j}{\epsilon^2}-\frac{V_1(X_1)}{\rho^2}.
\end{equation*}

\end{proof}


\bibliographystyle{plain}
\bibliography{scc,dft,additional}

\end{document}